\documentclass[11pt,a4paper]{article}
\usepackage{geometry}
\usepackage[authoryear,longnamesfirst]{natbib}

\usepackage{mathtools} %GA
\usepackage{amssymb,amsfonts}
\usepackage{amsthm}
\newtheorem{theorem}{Theorem}
\newtheorem{definition}[theorem]{Definition}
\newtheorem{lemma}[theorem]{Lemma}
\newtheorem{corollary}[theorem]{Corollary}

\newtheorem{assumption}[theorem]{Assumption}
\newtheorem{remark}[theorem]{Remark}
\usepackage{algpseudocode}
\usepackage{algorithm}
\usepackage{booktabs}
\usepackage{multirow}
\usepackage{xcolor}
\usepackage{hyperref}
\usepackage{natbib}
 \bibpunct[, ]{(}{)}{,}{a}{}{,}%
\usepackage{rotating}
\usepackage{fancyvrb}
\usepackage{enumitem}
\usepackage{cases}
\usepackage{nicefrac}

\newcommand{\mcO}{\mathcal{O}}
\newcommand{\mbR}{\mathbb{R}}
\newcommand{\mbP}{\mathbb{P}}
\newcommand{\mbE}{\mathbb{E}}
\newcommand{\mcF}{\mathcal{F}}
\newcommand{\mcB}{\mathcal{B}}
\newcommand{\mcX}{\mathcal{X}}

\newcommand{\mcN}{\mathcal{N}}
\newcommand{\mcU}{\mathcal{U}}
\newcommand{\mcA}{\mathcal{A}}

\newcommand{\mfU}{\mathfrak{U}}
\newcommand{\mfu}{\mathfrak{u}}
\newcommand{\mfs}{\mathfrak{s}}
\newcommand{\bsmfs}{\boldsymbol{\mathfrak{s}}}
\newcommand{\fhat}{\hat{f}}
\newcommand{\varHat}{\widehat{\text{Var}}}

\newcommand{\bftheta}{\boldsymbol{\theta}}
\newcommand{\UU}{\boldsymbol{U}}
\newcommand{\VV}{\boldsymbol{V}}
\newcommand{\XX}{\boldsymbol{X}}

\newcommand{\uu}{\boldsymbol{u}}
\newcommand{\vv}{\boldsymbol{v}}
\newcommand{\xx}{\boldsymbol{x}}

\newcommand{\var}{\textrm{Var}}
\newcommand{\RR}{\mathbb{R}}

\newcommand{\Wmax}{w_{\max}}
\newcommand{\Cmax}{c_{\max}}
\newcommand{\Kmax}{k_{\max}}
\newcommand{\Deltamax}{\Delta_{\max}}

\newcommand{\sigmamin}{\sigma_{\min}}
\newcommand{\sigmamax}{\sigma_{\max}}
\newcommand{\nbar}{\overline{n}}

\newcommand{\ovgamma}{\overline{\gamma}}
\newcommand{\ungamma}{\underline{\gamma}}
\newcommand{\mfL}{\mathfrak{L}}
\newcommand{\mcL}{\mathcal{L}}

\newcommand{\mre}{\mathrm{e}}

\newcommand{\term}{\mathfrak{N}_{F,\mu}}
\newcommand{\SE}{\textrm{SE}}
\newcommand{\SG}{\textrm{SG}}
\newcommand{\PO}{\textrm{PO}}
\newcommand{\EX}{\textrm{EX}}

\newcommand{\NS}{\textrm{NS}}
\newcommand{\bsa}{\boldsymbol{a}}
\newcommand{\bsb}{\boldsymbol{b}}

\newcommand{\bsy}{\boldsymbol{y}}
\newcommand{\bss}{\boldsymbol{s}}
\DeclareMathOperator*{\argmin}{argmin}

\begin{document}
\let\WriteBookmarks\relax
\def\floatpagepagefraction{1}
\def\textpagefraction{.001}
%\shorttitle{Stratified adaptive sampling for derivative-free stochastic trust-region optimization}
%\shortauthors{Amici, Shashaani, Jain}
%\begin{frontmatter}

%\title [mode = title]{Stratified adaptive sampling for derivative-free stochastic trust-region optimization}  
\title{Stratified adaptive sampling for derivative-free stochastic trust-region optimization}  

\author{ Giovanni Amici\footnote{Department of Industrial and Systems Engineering, North Carolina State University, 915 Partners Way, Raleigh, NC 27606, USA, {\em gamici@ncsu.edu}},
Sara Shashaani\footnote{Department of Industrial and Systems Engineering, North Carolina State University, 915 Partners Way, Raleigh, NC 27606, USA, {\em sshasha2@ncsu.edu}},
and 
Pranav Jain\footnote{Department of Industrial and Systems Engineering, North Carolina State University, 915 Partners Way, Raleigh, NC 27606, USA, {\em pjain23@ncsu.edu}}}

\maketitle

\begin{abstract}
There is emerging evidence that trust-region (TR) algorithms are very effective at solving derivative-free nonconvex  stochastic optimization problems in which the objective function is a Monte Carlo (MC) estimate. A recent strand of methodologies adaptively adjusts the sample size of the MC estimates by keeping the estimation error below a measure of stationarity induced from the TR radius. In this work we explore stratified adaptive sampling strategies to equip the TR framework with accurate estimates of the objective function, thus optimizing the required number of MC samples to reach a given \emph{$\epsilon$-accuracy} of the solution. 
%We prove a reduced sample complexity and explore inexpensive implementations in high dimension. 
We prove a reduced sample complexity, confirm a superior efficiency via numerical tests and applications, and explore inexpensive implementations in high dimension.
\end{abstract}

% \begin{graphicalabstract}
% \includegraphics{figs/cas-grabs.pdf}
% \end{graphicalabstract}

% \begin{highlights}
% \item Research highlights item 1
% \item Research highlights item 2
% \item Research highlights item 3
% \end{highlights}

% \begin{keywords}
% Stochastic optimization \sep
% Stratified sampling \sep
% Adaptive sampling \sep
% Trust-region optimization \sep
% Derivative-free optimization \sep
% Model fitting \sep
% Portfolio management
% \end{keywords}

{\it Keywords:} Stochastic optimization,
Stratified sampling,
Adaptive sampling,
Trust-region optimization,
Derivative-free optimization,
Model fitting,
Portfolio management.

\section{Introduction}\label{sec:intro}
\noindent
%\theme{Motivation to use stochastic opt DF algos}
In recent years, stochastic derivative-free optimization algorithms have become increasingly popular to solve problems governed by uncertainty.
Of particular interest are optimization problems in which only noisy evaluations of the objective function are available, requiring Monte Carlo (MC) methods to estimate the quantities of interest.
Examples of such problems can be found in engineering \citep{amaran2016simulation}, machine learning \citep{lan2020first},
%still ML-> \citealp{li2023stochastic}
and finance.
In many of these contexts, the gradient cannot be accessed, thus the problem becomes \emph{derivative-free} \citep{conn2009introduction, more2009benchmarking}.
Furthermore, the complex interactions between the relevant quantities can often make the problem nonconvex, preventing a strand of efficient convex optimization algorithms.

%\theme{Approaches to our kind of problem and recent works}
Methodologies that deal with the aforementioned issues include, stochastic direct search and model-based  algorithms with line-search methods being among the former and the trust-region methods among the latter. Along with recent advances in unconstrained settings (see, e.g., \citealp{ghadimi2013stochastic, rinaldi2024stochastic, shashaani2018astro,blanchet2019convergence}),
the incorporation of constraints is addressed for example in the line search context  \citep{cristofari2021derivative}, providing an ad hoc method to efficiently solve high dimensional problems under mild assumptions on the structure of the objective function.
%Methodologies that do not even attempt to exploit the structure of the objective function are random search algorithms \citep{andradottir2006overview, diouane2015globally}, which are more generally suited to deal with global optimization problems at the expense of being relatively complex from the computational viewpoint.
Methods that do not attempt to exploit the structure of the objective function are typically classified as random search algorithms \citep{andradottir2006overview, diouane2015globally}, which are well-suited for global optimization but often incur significant computational cost.
Whichever class of optimization algorithms is adopted, the estimation of the relevant uncertain quantities is of paramount importance in order to provide reliable solutions at an acceptable computational cost.
%When the problem requires Monte Carlo estimates, sampling becomes the key factor.

%\theme{Our work and contributions}
In this work we explore stratified adaptive sampling strategies to efficiently allocate the computational effort to the objective function estimation.
We show that equipping the stochastic optimization problem with a dynamic stratification of the state space of the underlying random vector can produce a lower \emph{sample complexity}, i.e., the amount of function evaluations required to reach a given solution accuracy.
Our work fits into the recent strand of adaptive sampling stochastic trust-region literature \citep{shashaani2016astro,shashaani2018astro,ha2025iteration,ha2025complexity}, but the sampling mechanism can be adopted for general stochastic derivative-free optimization algorithms.

While adaptive sampling has been thoroughly investigated by the aforementioned papers, the use of stratified sampling \citep{glasserman2004monte} in the optimization context is more recent (see, e.g., \citealp{jain2023wake, jain2024simulation}), but with the main difficulty of identifying strata during optimization.
In contrast, in this work we focus on a dynamic uniform stratification coupled with an inverse transform map to sample from the distribution of interest.
This study aims to quantify the benefits of stratification on the convergence and computational complexity of optimization algorithms, which is a key contribution of our work.
We demonstrate that well-designed stratification can significantly enhance convergence rates and reduce the number of simulations required to achieve optimal solutions. Our analysis extends the classical stratified sampling variance reduction characterization that was only for the bounded-support random variables to a broader and more generalized cases of random vectors.
Moreover, in the interest of generality we allow the distribution of the noise to depend on the decision vector, making our framework remarkably flexible. 
To support our theoretical results, we also provide some numerical implementations of our algorithm, showing its superior performance with respect to suitable benchmarks.
Furthermore, we discuss a parsimonious data-driven sampling procedure (alternative to our main sampling procedure) that one can adopt when the problem is high dimensional and large datasets are available, such as in machine learning applications.

%\theme{Applications (mostly in finance)}
\subsection{Applications}
We envisage a range of applications for the optimization problem analyzed in this work.
Perhaps the most obvious real-world problems under our mild assumptions on the objective function regard simulation-based applications and black-box optimization problems.
However, there are other problems which can fit into our framework.
For example, in the context of model estimation, the expectation-maximization algorithm \citep{dempster1977maximum} solves a sequence of stochastic optimization problems in which the risk factor depend on the current iterate.
Moreover, in the maximum likelihood context, when the stochastic process to be estimated is multidimensional and the associated density function is only available up to a Fourier transform inversion (see, e.g., \citealp{ballotta2017multivariate,amici2025multivariate,amici2025multivariateLevy,bianchi2025welcome}), it may be convenient to use sampling techniques to approximate the multidimensional integral involved in the Fourier inversion
%, which is interpreted as an expectation
(e.g., importance sampling for high-dimensional option pricing as in \citealp{ballotta2022powering}).

Other applications can be found in the context of model calibration.
For example, in option pricing applications \citep{cont2010model}, the pricing formula is typically formulated as an expectation and may require sampling procedures if the underlying assets are multiple, i.e., when the random component of the optimization problem is multidimensional. 
Examples of this kind include calibration to index options (see e.g. \citealp{neufeld2023model, bondarenko2024option}), when the pricing model is used to describe the behavior of the index components and their dependence.
Many of the aforementioned estimation and calibration problems in finance are in fact characterized by non-convex objective functions (see, e.g., \citealp{cont2004nonparametric}) and are unconstrained in the parameters of the model, such as in the case of geometric Brownian motion or the variance gamma process \citep{schoutens2003levy}. 
Moreover, objective functions are typically not available analytically, making derivative-free algorithms particularly appealing.

Moreover, our optimization method can be a suitable candidate for certain allocation problems.
In these applications the objective function could be nonconvex, for example, when it includes higher order moments of the distribution of interest or when it is formulated as an (expected) utility function.
As we allow the random component of the optimization problem to depend on the decision vector, we emphasize that this characteristic is present in applications in which the \emph{price impact} effect of decisions (e.g., trades) is modeled (see, e.g., \citealp{webster2023handbook}).
These applications can be found in contexts of high-frequency trading \citep{schied2019market, hey2025trading}, portfolio optimization \citep{chen2014analytical, iancu2014fairness, edirisinghe2023optimal}, and operations management \citep{wei2014optimal, cruise2019control}.

In the aforementioned applications the probability distribution of the random vector of interest is typically known, at least up to an estimation.
Accordingly, it is possible to simulate the random vector by first sampling from a unit hypercube, and then applying an inverse map (e.g., the Rosenblatt transformation, \citealp{rosenblatt1952remarks}) to sample from the distribution of interest.
Thus, employing stratification strategies can be relatively simple, when applied to the hypercube, and reduce the overall complexity of the algorithm, as revealed throughout the paper.
In contrast, in situations in which estimating a distribution is either expensive or not possible, it may be convenient to use data-driven sampling techniques, which are in some form more consistent to black-box optimization problems; we will discuss a data-driven alternative formulation of the algorithm towards the end of the paper.

%\theme{Structure of the paper}
The remainder of the paper is organized as follows.
In Section \ref{sec:problem_statement}, we describe the optimization problem of interest.
Section \ref{sec:preliminaries} reports some preliminary notions on the trust-region framework under which we propose the new sampling procedure.
In Section \ref{sec:sastro}, we provide the key results on the convergence and the complexity of our algorithm, showcasing a comparison with related benchmarks.
In Section \ref{sec:applications}, we illustrate our main algorithm, some numerical numerical experiments, and a data-driven extension, and Section \ref{sec:conclusions} concludes.

\subsection{Problem Statement} \label{sec:problem_statement}
\noindent
%\theme{Basic problem}
We are interested in solving the following optimization problem.
Let $F:\RR^d\times\RR^q\to\RR$ be a function taking as input a set of parameters $\bftheta\in \Theta \subseteq \RR^d$ and a random vector $\XX$ defined in the probability space $(\Omega, \mathcal{F}, \mathbb{P})$ and taking values in $\mcX \subseteq \RR^q$.
Our objective is then to find 
\begin{align}\label{eq:opt}
    \bftheta^* = \argmin_{\bftheta\in \Theta \subseteq \RR^d
    } \left\{f(\bftheta):=\mathbb{E}[F(\bftheta,\XX)]\right\}.
\end{align}
Here, we assume $f$ to be possibly nonconvex and continuously differentiable, but we do not assume access to the function derivatives and hence the optimization problem becomes derivative-free.
Moreover, we allow $\XX$ to depend on the decision vector $\bftheta$ (for example, the distribution of $\XX$ may be parametrized by $\bftheta$).
Furthermore, we assume that the expectation in \eqref{eq:opt} can only be computed up to a MC estimation.
%, and each MC call is freshly generated at each evaluation of the objective function. %->quantum computing
These characteristics make the study of problem \eqref{eq:opt} appealing for a variety of applications, as reported in Section~\ref{sec:intro}.

%\theme{Motivation for adaptive sampling}
The expected quantity in a data-driven setting where $n$ copies of the $q$-dimensional input data $\{\xx_i\}_{i=1}^{n}$ are available can turn into a deterministic problem with sample average approximation, i.e., 
\begin{equation*}%\label{eq:opt_SAA}
    \bftheta_n^* = \argmin_{\bftheta\in \Theta \subseteq \RR^d
    } \left\{\fhat (\bftheta, n) :=\frac{1}{n} \sum_{i=1}^n F(\bftheta,\xx_i)\right\}.
\end{equation*}
In this setting, evaluating the objective function at every instance of $\bftheta$ using the entire dataset can be computationally expensive. Additionally, with a fixed sample size $n$, $\bftheta_n^*$ will never converge to $\bftheta^*$, which is to say that even if we find $ \bftheta_n^*$ exactly, it may not be an optimal parameter for a set of unseen data.

It is possible to deal with these shortcomings by sampling a random amount $N$ of data points at each evaluation of the objective function. This random sample size will have to go to infinity if we have any hopes of our iterative optimization algorithm converging to $\bftheta^*$. An \emph{adaptive sampling} strategy would govern how quickly this sample size should go to infinity based on how quickly the algorithm will make progress towards optimality. Since the standard error of our estimate is proportional to $N^{-1/2}$, a possible adaptive sampling idea may be to force the standard error to be smaller than a measure of stationarity in the optimization algorithm. Specifically, when the algorithm requests an estimate of the objective function at an iterate $\bftheta_k$, the adaptive sampling rule would choose the sample size as 
\begin{align}\label{eq:adaptive_sampling_rule_PRELIMINARIES}
    N_k=N(\bftheta_k)=\min\left\{n:\ \sqrt{\frac{\widehat{\var}(F(\bftheta_k,\XX) \mid \bftheta_k )}{n}}\leq \text{ measure of stationarity at } \bftheta_k \text{ deflated with }k\right\},
\end{align} 
% \begin{equation}
%    N_k=\min\left\{n \:\geq N_{k-1}: \sqrt{\widehat{\text{Var}}(\tilde{f}(\bftheta_k,n))}\leq \phi(\Delta_k) \right\},
% \end{equation}
where $\widehat{\text{Var}}(F(\bftheta_k,\XX))$ is an estimate of the variance of $F(\bftheta_k,\XX)$. The \emph{deflation} of the stationarity measure with a deterministic decreasing sequence in $k$ is warranted given the stochasticity of the variance estimate. That is, the rate at which the estimation error drops to zero should be somewhat faster than the rate of convergence to stationarity.

The estimate of the variance helps with adjusting the sample size based on the local variability of the objective function. This sequential sampling procedure was first introduced for the ASTRO-DF (Adaptive Sampling Trust-Region Optimization--Derivative-Free) algorithm~\citep{shashaani2018astro}, yielding a sample size that is a stopping time and random. Existing work has used a \emph{dynamic sampling} process that presumes access to the true variance in stochastic gradient and line-search methods~\citep{bollapragada2018adaptive,franchini2023line} yielding a deterministic sample size. In all these approaches, utilizing variance reduction techniques that can reduce the estimation error with less number of samples can be an effective tool to improve the \emph{sample complexity}. Among the recent work in this direction using adaptive importance sampling see \citep{camellini2025line}. Sample complexity is measured in terms of total number of MC samples before reaching an $\epsilon$-stationary solution; for the first-order unconstrained stationarity this is when $\|\nabla f(\bftheta_k)\|<\epsilon$. The first iteration to reach this condition, termed as $T_\epsilon$, is a stopping-time random variable, as $\{\bftheta_k\}_{k\geq0}$ is a random sequence. An analysis of the sample complexity, therefore, would entail providing a probabilistic sense of how large the number of total samples would be as a function of $\epsilon$.
In a derivative-free setting, let $\bftheta_k$ be the current iterate, $\{\bftheta_k^{(1)}, \bftheta_k^{(2)}, \ldots, \bftheta_k^{(p)}\}, \ p\in\mathbb{N},$ be the additional interpolation points used to construct the $k$-th trust region model, and $\bftheta_k^{s}$ denote the $k$-th candidate solution. 
%that will assist the algorithm in finding the candidate solution
By convention we will sometimes denote $\bftheta_k$ as $\bftheta_k^{(0)}$ and $\bftheta_k^s$ as $\bftheta_k^{(p+1)}$.
Then, the sample complexity at $\epsilon$-stationarity involves characterization of the random variable
\begin{align}\label{eq:total_samples_PRELIMINARIES}
    W_{\epsilon} 
    := \sum_{k=0}^{T_\epsilon} \left( 
    \sum_{i=0}^{p} N\left(\bftheta_k^{(i)}\right)
    + N(\bftheta_k^s)
    \right)
\end{align}
where $N(\bftheta_k^{(i)})$ is the sample size of the $i$-th point estimation at iteration $k$, and $N(\bftheta_k^{s})$ is the sample size at the candidate evaluation point (henceforth denoted $N_k^{(i)}$ and $N_k^{s}$), as a function of $\epsilon$. 
For example, a sublinearly converging algorithm would obtain $W_{\epsilon} = \mcO(\epsilon^{-a}),\ a>0$ with high probability or with probability one or in expectation.

\subsection{Our Contributions}
%\theme{Our method and recap of results}
In this paper, we analyze the effect of stratified sampling on the sample complexity of ASTRO-DF; we develop a variant of ASTRO-DF termed SASTRO-DF that features a stratified adaptive sampling.
%%%%%%%%%%
%ASTRO-DF has been proven to attain $W_\epsilon=\mathcal{O}(\epsilon^{-6}\log \epsilon^{-1})$ with probability one (a well-known result for stochastic nonconvex derivative-free optimization algorithms) and when $F(\bftheta,\cdot)$ is Lipschitz continuous, $W_\epsilon=\mathcal{O}(\epsilon^{-4}\log \epsilon^{-1})$ by means of common random numbers (CRN) \citep{ha2025complexity}. 
Recent papers on the ASTRO-DF explored the complexity of the algorithm equipping it with a common random numbers (CRN) scheme to generate samples of $\XX$.
CRN uses the same MC draws $\XX_{1},\ldots,\XX_{N_k}$ for the evaluation of the objective function on points $\bftheta_k^{(i)},\ i=0, \ldots, p+1$, but in some applications, one may not have the ability to use the same samples. Additionally, the ability to improve sample complexity under CRN relies on \emph{sub-exponentiality} of $F(\bftheta,\XX)$ for all $\bftheta\in\Theta$ to invoke Bernstein tail bounds. In this paper, we let the MC instances change across points and provide an analysis that does not rely on subexponential tails for $F(\bftheta,\XX)$ and employs the more general Chebyshev tail bounds, giving a broad range of applicability to our algorithm.
%including the optimization problems in finance. 
Our numerical results in Section~\ref{sec:applications} show that ASTRO-DF with Bernstein-bound logarithmic deflators fall short of solving simple portfolio optimization in low dimensions when given limited computational budget significantly underperforming the stratified adaptive sampling with Chebyshev-bound polynomial deflators.

Most importantly, we explore stratification strategies to optimize the number of samples required to reach $\epsilon$-stationarity.
When the distribution of $\XX$ satisfies some regularity conditions, we show that SASTRO-DF via $q$-dimensional stratification  for $q< 4$ attains a lower sample complexity than the standard no-stratification strategy 
implicitly employed in the recent papers on the ASTRO-DF algorithm, under the same conditions.
%}
The key factor for such better performance is that we equip the optimization algorithm with a more accurate estimator, which in turn allows for a fewer samples to satisfy  \eqref{eq:adaptive_sampling_rule_PRELIMINARIES}.
The improvement in efficiency is remarkable in low dimensional $\XX$ (noise space) and more  moderate when the dimension of $\XX$ increases (regardless of the dimension of $\bftheta$).
We point out that $q$-dimensional stratification is in general not suited for large $q$ as a consequence of the exponential increase of the number of strata.
A more parsimonious stratification in high dimension can be performed via Latin hypercube sampling \citep{mckay1979comparison,stein1987large}, however the advantage of this technique is limited to the case in which $\XX$ has independent components (see a discussion in Chapter 4.4 \citealp{glasserman2004monte}).

Furthermore, to extend the use case of SASTRO-DF for machine learning problems with large datasets, we discuss a \emph{data-driven} one-dimensional stratification strategy in which each $q$-dimensional entry of a given dataset is mapped to a given point of the unit interval.
As opposed to the $q$-dimensional stratification, in this situation it is not possible to create a continuous bijective map between the space of the input data $\XX$ and the space in which the stratification is applied.
Thus, this sampling strategy lives in a purely discrete distribution-free framework in which only subsets of the available data can be sampled.
%We note that this sampling mechanism could be extremely relevant for applications in which large datasets are available.

\section{Preliminaries}\label{sec:preliminaries}
\noindent
To solve the problem defined in Section \ref{sec:problem_statement} we rely on a trust-region framework mainly inspired by the ASTRO-DF algorithm of \citep{shashaani2016astro,shashaani2018astro}.
In particular, we build a local surrogate model at each iteration and adaptively determine the sample size for sample average approximation.
If $\bftheta_k$ is the iterate at iteration $k$, a local model $M_k(\cdot)$ is built in a ball $\mcB_k(\bftheta_k,\Delta_k) = \{\bftheta : \|\bftheta - \bftheta_k \|_2 \leq \Delta_k \}$  around $\bftheta_k$ with $\Delta_k$ being the trust-region radius.
In the zeroth order setting, this local model may be constructed via second order polynomial interpolation using $2d$ interpolation points in $\mcB_k$ in addition to the center, i.e., the iterate $\bftheta_k$ (see Definition 2.1 in \citep{ha2025iteration} for this particular case and \citep{roberts2025introduction} for more general constructions).
In the following we report the formal definition of this type of interpolation model.

%%%%%
\begin{definition}\label{def:polynomial_model}(Stochastic polynomial interpolation model (SPIM))\\
Let $\mcB(\bftheta_k,\Delta_k) = \{\bftheta\in\Theta \ : \ \|\bftheta-\bftheta_k\| \leq \Delta_k\}$ be the trust region at iteration $k$, where $\bftheta_k$ is the current solution and $\Delta_k$ is the trust-region radius.
Let $\Theta_k = (\bftheta_k^{(i)},i=0,1,\ldots,p)$ be a set of design points at iteration $k$, where $\bftheta_k^{(0)}=\bftheta_k$, and let $\tilde{\boldsymbol{f}}_k = \left(\tilde{f}_k(\bftheta_k^{(0)}, N_k^{(0)}), \tilde{f}_k(\bftheta_k^{(1)}, N_k^{(1)}), \ldots, \tilde{f}_k(\bftheta_k^{(p)}, N_k^{(p)}) \right)^{\top}$ be the corresponding vector of estimated objective functions.
Let $\Phi = (\phi_0,\phi_1,\ldots,\phi_r)^{\top}$ be a vector whose components compose a polynomial basis on $\RR^d$, $\beta_k=(\beta_{k,0}, \beta_{k,1}, \ldots, \beta_{k,r})^{\top}$ be a vector of $\RR$-valued coefficients, and 
\begin{align}\label{eq:matrix_M_general}
    \mathcal{M}(\Phi,\Theta_k) = \begin{bmatrix}
    \Phi(\bftheta_k^{(0)})^\top \\
    \Phi(\bftheta_k^{(1)})^\top \\
    \vdots \\
    \Phi(\bftheta_k^{(p)})^\top
    \end{bmatrix}.
\end{align}
Then, by letting $p=r$, the SPIM of the true objective function $f$ is given by 
    $M_k(\bftheta) = \sum_{i=0}^p \beta_{k,i} \phi_i(\bftheta)$
where $\beta_k$ solves $\mathcal{M}(\Phi,\Theta_k) \ \beta_k = \tilde{\boldsymbol{f}}_k$.
\\
\end{definition}

% %%%%%
% \begin{definition}\label{def:polynomial_model_L}(Linear - stochastic polynomial interpolation model (L-SPIM))\\
% Keep the same notation of Definition \ref{def:polynomial_model}.
% Let $p=r=d$, the design points be
% $
% \Theta_k^{\text{L}}
% =
% \left\{
% \bftheta_k^{(0)},\;
% \bftheta_k^{(0)} + \Delta_k e_1,\;
% \dots,\;
% \bftheta_k^{(0)} + \Delta_k e_d
% \right\}
% $,
% where \(\{e_j\}_{j=1}^d\) denotes the canonical basis of \(\mathbb{R}^d\),
% and the polynomials be defined as
% $
% \phi^{\text{L}}_0(\bftheta)=1, \
% \phi^{\text{L}}_j(\bftheta)=\theta_j-\theta_{k,j}=\Delta_k, \
% j=1,\ldots,d,
% $
% where $\theta_j$ and $\theta_{k,j}$ are the $j$-th scalar components of $\bftheta$ and $\bftheta_{k}=\bftheta_k^{(0)}$, respectively.
% Then, matrix $\mathcal{M}$ defined in  \eqref{eq:matrix_M_general} is specified as
% \begin{align}
% \mathcal{M}(\Phi^{\text{L}},\Theta_k^{\text{L}}) =
% \begin{bmatrix}
% 1 & 0 & 0 & \cdots & 0 \\
% 1 & \Delta_k & 0 & \cdots & 0 \\
% 1 & 0 & \Delta_k & \cdots & 0 \\
% \vdots & \vdots & \vdots & \ddots & \vdots \\
% 1 & 0 & 0 & \cdots & \Delta_k \\
% \end{bmatrix}
% \end{align}
% Then, the L-SPIM of the true ojective function $f$ is given by
% \begin{align}\label{eq:polynomial_model_L}
%     M_k^{\text{L}}(\bftheta) = \sum_{i=0}^p \beta_{k,i}^{\text{L}} \phi_i^{\text{L}}(\bftheta),
% \end{align}
% where $\beta_k^{\text{L}}$ solves $\mathcal{M}(\Phi^{\text{L}},\Theta_k^{\text{L}}) \ \beta_k^{\text{L}} = \tilde{\boldsymbol{f}}_k$.
% \\
% \end{definition}

%%%%%
\begin{definition}\label{def:polynomial_model_DQ}(Diagonal quadratic - stochastic polynomial interpolation model (DQ-SPIM))\\
Keep the same notation of Definition \ref{def:polynomial_model}.
Let $p=r=2d$, the design points be
$
\Theta_k^{\text{DQ}}
=
\{
\bftheta_k^{(0)},\;
\bftheta_k^{(0)} \pm \Delta_k e_1,\;
\dots,\;
\bftheta_k^{(0)} \pm \Delta_k e_d
\}
$,
where \(\{e_j\}_{j=1}^d\) denotes the canonical basis of \(\RR^d\),
and the polynomials be defined as
$
\phi^{\text{DQ}}_0(\bftheta)=1, \
\phi^{\text{DQ}}_j(\bftheta)=\theta_j-\theta_{k,j}=\Delta_k, \
\phi^{\text{DQ}}_{d+j}(\bftheta)=\frac{1}{2}(\theta_j-\theta_{k,j})^2=\frac{1}{2}\Delta_k^2, \
j=1,\ldots,d,
$
where $\theta_j$ and $\theta_{k,j}$ are the $j$-th scalar components of $\bftheta$ and $\bftheta_{k}=\bftheta_k^{(0)}$, respectively.
Then, matrix $\mathcal{M}$ defined in  \eqref{eq:matrix_M_general} is specified as
\begin{align*}
\mathcal{M}(\Phi^{\text{DQ}},\Theta_k^{\text{DQ}}) =
\begin{bmatrix}
1 & 0 & 0 & \cdots & 0 & 0 & 0 & \cdots & 0 \\
1 & \Delta_k & 0 & \cdots & 0 & \tfrac{\Delta_k^2}{2} & 0 & \cdots & 0 \\
1 & 0 & \Delta_k & \cdots & 0 & 0 & \tfrac{\Delta_k^2}{2} & \cdots & 0 \\
\vdots & \vdots & \vdots & \ddots & \vdots & \vdots & \vdots & \ddots & \vdots \\
1 & 0 & 0 & \cdots & \Delta_k & 0 & 0 & \cdots & \tfrac{\Delta_k^2}{2} \\
1 & -\Delta_k & 0 & \cdots & 0 & \tfrac{\Delta_k^2}{2} & 0 & \cdots & 0 \\
1 & 0 & -\Delta_k & \cdots & 0 & 0 & \tfrac{\Delta_k^2}{2} & \cdots & 0 \\
\vdots & \vdots & \vdots & \ddots & \vdots & \vdots & \vdots & \ddots & \vdots \\
1 & 0 & 0 & \cdots & \underset{d}{\underbracket{-\Delta_k}} & 0 & 0 & \cdots & \underset{2d}{\underbracket{\tfrac{\Delta_k^2}{2}}}
\end{bmatrix}
\end{align*}
Then, the DQ-SPIM of the objective function $f$ is given by
    $M_k^{\text{DQ}}(\bftheta) = \sum_{i=0}^p \beta_{k,i}^{\text{DQ}} \phi_i^{\text{DQ}}(\bftheta)$,
where $\beta_k^{\text{DQ}}$ solves $\mathcal{M}(\Phi^{\text{DQ}},\Theta_k^{\text{DQ}}) \ \beta_k^{\text{DQ}} = \tilde{\boldsymbol{f}}_k$.
\\
\end{definition}

%%%%%%
The local model is then minimized in $\mcB_k$ to give a candidate solution $\bftheta^s_{k}$.
This candidate solution is accepted ($\bftheta_{k+1} = \bftheta^s_{k}$) if sufficient reduction in objective function is achieved at $\bftheta^s_{k}$ and if the model-gradient (tracing the true-gradient) is not much smaller than the trust-region radius (a need arising in derivative-free settings). If the candidate solution is accepted, the trust-region radius expands; otherwise, the trust-region radius shrinks, and $\bftheta_{k+1} = \bftheta_k$, resulting in a new model in a smaller region at the same incumbent for the next iteration.

% In our setting, the trust-region model $M_k$ is a fully linear model of the estimated objective function $\tilde{f}$ (see Definition 10.3 in \citealp{conn2009introduction}). %useful for convergence proof in Corollary \ref{cor:convergence_of_the_algo}
% It follows that, using the aforementioned zeroth order setting, $M_k$ is a \emph{stochastically fully linear model} of  $f$ in the trust region (see Definition 2.3 and Remark 2 in \citealp{ha2025iteration}), which is a useful property to ensure the convergence of the algorithm, once the quality of the estimator $\tilde{f}$ has been assessed.

To update the trust region we adopt the following mechanism, similar to ASTRO-DF.
First, at any iteration $k$, in order to accept a new candidate solution the \emph{acceptance ratio} $\rho_k$ must be greater that or equal to a constant $\eta\in(0,1)$, that is,
\begin{equation}\label{eq:trust_region_rho}
    \rho_k := \frac{\tilde{f}(\bftheta_k,N_k) - \tilde{f}(\bftheta_k^s,N_k^s)}{M_k(\bftheta_k) - M_k(\bftheta_k^s)} \geq \eta.
\end{equation}
%either equal to the Cauchy step or to a better performing step.
%
Second, we update the trust-region radius as
\begin{equation}\label{eq:trust_region_update}
    \Delta_{k+1} =
    \begin{cases}
        \min(\ovgamma \Delta_k, \Deltamax) & \text{if\: $\rho_k\geq\eta$ and $\Delta_k\leq \tilde{\eta}\|\nabla M_k(\bftheta_k)\|$}, \\
        \ungamma \Delta_k & \text{otherwise},
    \end{cases}
\end{equation}
for some $0 < \ungamma < 1 < \ovgamma$, $\Deltamax>0$, and $\tilde{\eta}>0$.
All these mechanisms are formalized and summarized in Algorithm~\ref{alg:sastro-df}.

In addition, consistently with the trust-region framework we formulate the adaptive sampling rule to incorporate the trust-region radius $\Delta_k$ and iteration $k$, that is, 
\begin{align}\label{eq:adaptive_sampling_rule_trust_region_PRELIMINARIES}
   N_k^{(i)} = \min\left\{n \in \mathbb{N} \
   %\geq N_{k-1} \ 
   : \ \sqrt{ \widehat{\var} \left(\tilde{f}(\bftheta_k,n)\right)}\leq \varphi(k, \Delta_k) \right\},
   \qquad k\geq0, \quad i=0,1,\ldots,p+1,
\end{align}
as $\Delta_k$ is related to how close the current solution is to stationarity and $\varphi$ is a function constructed in order to guarantee convergence, to be specified later.
The stopping rule in  \eqref{eq:adaptive_sampling_rule_trust_region_PRELIMINARIES} ensures that sample sizes increase over iterations, which enables early exploration and better estimation towards the end of the algorithm.

\section{The SASTRO-DF}\label{sec:sastro}
\noindent

In this section we provide results concerning the variance of the stratified sampling estimator \eqref{eq:stratified_estimator}, the convergence of the SASTRO-DF algorithm, and its complexity.
%(we note that the iteration complexity is the same as the ASTRO-DF of \citealp{shashaani2016astro,shashaani2018astro}).
To this purpose, we introduce the following definitions and assumptions that will be used throughout the paper.

\begin{assumption}\label{ass:lipschitz}
Function $f$ in  \eqref{eq:opt} is bounded below and its gradient is Lipschitz continuous on $\Theta$.
\end{assumption}

\begin{assumption}\label{ass:smooth}
Function $F(\cdot,\XX)$ in  \eqref{eq:opt} is continuously differentiable in $\XX$ and has bounded variance for all $\bftheta\in\Theta$.
\end{assumption}

\begin{assumption}\label{ass:map}
Let $\UU$ be a $q$-dimensional uniform random variable with support $\mcU$.
There exists a continuous and bijective map
$\mu: \mcX\to \mcU$
such that $\mu$ and $\mu^{-1}$ are continuously differentiable almost everywhere. 
In other words, $\mu$ defines a diffeomorphism between the supports $\mcX\subseteq\RR^q$ and $\mcU$ of $\XX$ and $\UU$, respectively \citep{milnor1997topology} almost everywhere in the support.
For ease of exposition, we consider the unit hypercube of the form $\mcU=(0,1]^q$, but it can be adapted to other expressions depending on the particular form of $\mcX$.
\end{assumption}

\begin{remark}
    Without loss of generality, we assume that the map $\mu$ introduced in Assumption \ref{ass:map} is the Rosenblatt transformation \citep{rosenblatt1952remarks}, although it may be constructed differently (see, e.g., \citealp{papamakarios2021normalizing}).
    Thus, it satisfies
    \begin{align*}
        X_1 = \mu_1^{-1}(U_1), \
        X_2 = \mu_2^{-1}(U_2\mid X_1), \
        \ldots, \
        X_q = \mu_q^{-1}(U_q\mid X_1,\ldots,X_{q-1}),
    \end{align*}
    where $\mu_i^{-1}(U_i\mid \cdot)$ denotes the inverse conditional cumulative distribution function of $X_i, \ i=1,\ldots,q$.
    As mentioned in the introduction, we allow the distribution of $\XX$ to depend on $\bftheta$, in which case the map may be denoted as  $\mu_{\bftheta}$; however we omit the subscript for ease of notation.
\end{remark}

% \begin{definition}\label{def:filtration}
%     Let $\{\XX^{[m]}, \ m\geq1\}$ be a sequence of random vectors in the probability space $(\Omega, \mcF,\mbP)$ as defined in  \eqref{eq:opt}.
%     Define the corresponding natural filtration as $\mathcal{G}_n=\sigma(\XX^{[m]}, \ m\leq n), \ n\geq 1$.
%     Let $\{N_k, \ k\geq0\}$ be a stopping time with respect to $\{\mathcal{G}_{N_{k-1}+n}, \ n\geq0\}$, where $N_{-1}:=n_0\in\mathbb{N}$ is an initialization constant.
%     Then, we define the iteration filtration as $\mcF_k:=\mathcal{G}_{N_k}, \ k\geq0$, containing all the information up to the end of iteration $k$, where the first filtration $\mcF_{-1}$ contains the initialization parameters.
% \end{definition}

\begin{definition}\label{def:filtration}
    Let $\{\UU^{[m]}, \ m\geq1\}$ be a sequence 
    %of uniform $[0,1]^q$ 
    random vectors in the probability space $(\Omega, \mcF,\mbP)$.
    Define the corresponding natural filtration as $\mathcal{G}_n=\sigma(\UU^{[m]}, \ m\leq n), \ n\geq 1$.
    Let $\{\lambda_k, \ k\geq-1\}$ be a deterministic sequence of positive real numbers.
    Let $\{N_k^{(i)}, \ k\geq0, \ i=0,1,\ldots,p+1\}$ be a stopping time with respect to $\{\mathcal{G}_{\lceil \lambda_{k-1} \rceil + n}, \ n\geq0\}$.
    Then, we define the iteration filtration as
    $\mcF_k:=\mathcal{G}_{\max_{i} \left\{ N_k^{(i)} \right\}}, \ k\geq0$,
    %$\mcF_k:=\mathcal{G}_{\max_{k,i} \left\{ N_k^{(i)} \right\}}, \ k\geq0$, for any $i$,
    containing all the information up to the end of iteration $k$, where the first filtration $\mcF_{-1}$ contains the initialization parameters.
\end{definition}

\begin{remark}\label{rem:filtration}
    %Consistently with Definition \ref{def:filtration}, f
    For any interpolation or candidate point $i$, the iterate $\{\bftheta_k^{(i)},\ k\geq1\}$ and the trust-region radius $\{\Delta_k, k\geq1\}$ are $\mcF_{k-1}$-adapted processes, whereas the adaptive sample size $\{N_k^{(i)},k\geq1\}$ and the sequence $\{\UU^{[m]}, \ N_{k-1}^{(i)}< m\leq N_{k}^{(i)}\}$ are $\mcF_{k}$-adapted processes.
\end{remark}

\begin{remark}
    In our setting, at each iteration $k$ we reuse the set of uniform samples generated up to iteration $k-1$ and only generate $\sum_{i=0}^{p+1} N_k^{(i)}-N_{k-1}^{(i)}$ new samples, which is consistent with the information flow illustrated in Definition \ref{def:filtration}.
    However, we remark that this does not prevent the inverse transform map to change across iterations, producing a completely new set of instances of $\XX$.
\end{remark}

% \begin{remark}
%     In the interest of generality, in Definition \ref{def:filtration} we have assumed that Monte Carlo samples are \lq\lq refreshed\rq\rq\ across iterations, so that after $k$ iterations we have generated $\sum_{l=0}^{k-1} N_l$ sample deviates.
%     This procedure is especially needed in cases in which the distribution of $\XX$ depends on the decision vector $\bftheta_k$.
%     % However, when the optimization problem allows for it, our framework could be adapted to a more parsimonious strategy in which at each iterations only an incremental number of samples deviates is generated.
%     % That is, the total number of Monte Carlo points after $k$ iterations would be $N_0 + \sum_{l=1}^{k-1} (N_1-N_{l-1}) = N_{k-1}$.
% \end{remark}

% \begin{remark}\label{rem:filtration}
%     In relation to Definition \ref{def:filtration}, the iterates $\{\bftheta_k,k\geq1\}$, the trust-region radii $\{\Delta_k, k\geq1\}$, \magenta{and the adaptive sample sizes $\{N_k,k\geq1\}$}, are $\mcF_{k}$-adapted processes.
%     Thus, given the information induced by $\mcF_{k}$, the only source of randomness comes from the new $N_k$ realizations of the random vector $\XX$ used to estimate the objective function.
% \end{remark}

\begin{definition}
    At each iteration $k$ and interpolation or candidate point $i=0,1\ldots,p+1$, for a given dimension $q$, and for some $l_k^{(i)}\in\mathbb{N}$, we define the sets of strata boundaries and strata hypercubes as
    \begin{align}
    & \mfL_k^{(i)} = \left\{0,\frac{1}{l_k^{(i)}},\ldots, \frac{l_k^{(i)}-1}{l_k^{(i)}} \right\}^q\!\!\!,
    \label{eq:hypercube_strata_leftend}
    \\
    & \mfU_k^{(i)}
    =\left\{
    \prod_{j=1}^q \left( u_j,\, u_j+\frac{1}{l_k^{(i)}} \right]
    \ : \
    \forall (u_1,\ldots,u_q) \in \mfL_k^{(i)}
    \right\},
    \label{eq:hypercube_strata}
    \end{align}
    both with cardinality $L_k^{(i)}:= \left(l_k^{(i)}\right)^q$, which is the total number of strata at iteration $k$ and point $(i)$.
    %\blue{In the remainder of the paper we will sometimes omit the superscript $(i)$ for ease of notation, without affecting the analysis.}
\end{definition}

\begin{definition}
    For any element $\mfu \in \mfU_k, \ k\geq 0$, we define the corresponding stratum in $\mcX$ as
        $\mcX_{\mfu} = \mu^{-1}(\mfu)$, 
    where $\mcX_{\mfu}\cap \mcX_{\mfu'}=\emptyset, \forall \mfu\neq\mfu'$, and $\cup_{\mfu} \mcX_{\mfu}=\mcX$.
\end{definition}

\begin{assumption}\label{ass:strata}
For any dimension $q$, iteration $k$, and interpolation or candidate point $i$, it holds that $L_k^{(i)}\in\mcL_q:=\{\ell\in\mathbb{N} \ : \ \ell^{1/q}\in\mathbb{N}\}$, where $\mcL_q$ denotes the admissible number of strata.
%%% strong assumption: we make it in the interest of exposition clarity, but also in practice it would not be easy to relax such assumption.
The number of samples per strata is the same across strata and iterations, i.e., $N_{k, \mfu}^{(i)} = N_k^{(i)}/L_k^{(i)} := \nbar\in\mathbb{N}$, for all $\mfu\in\mfU_k, \ k\geq0, \ i=0,1,\ldots,p+1$. Thus, it holds that
\begin{equation}\label{eq:set_N}
    N_k^{(i)} \in \mcN_q := \left\{n\in\mathbb{N} \ : \ \frac{n}{\ell}=\nbar \in\mathbb{N}, \:\forall \ell\in\mcL_q  \right\},
    \quad k\geq0, \ i=0,1,\ldots,p+1.
\end{equation}
Also, each stratum has equal probability, corresponding to $1/L_k^{(i)} = \nbar/N_k^{(i)},\: \forall k\geq0, \ i=0,1,\ldots,p+1$.
\end{assumption}

Given the above, for a given decision vector $\bftheta$, samples size $n$, and set of strata $\mfU$ with $\ell=|\mfU|$, we are interested in a stratified sampling estimator with equiprobable strata and proportional allocation of MC samples.
Denoting the stratified sampling estimator with $\tilde{f}$, we have
\begin{align}
    \tilde{f}(\bftheta,n) 
    & = \sum_{\mfu\in\mfU} \mbP(\XX\in\mcX_{\mfu})\ \hat{f}_{\mfu} (\bftheta,n_{\mfu}) = \frac{1}{\ell} \sum_{\mfu\in\mfU} \hat{f}_{\mfu} (\bftheta, \nbar)\label{eq:stratified_estimator}
\end{align}
where $n_{\mfu}$ is the number of samples in stratum $\mfu$ and $\hat{f}_{\mfu}(\bftheta, \nbar)$ is the sample average in stratum $\mfu$ using $\nbar$ points.
In addition, its variance is
\begin{align}\label{eq:stratified_estimator_variance}
    \var(\tilde{f}(\bftheta, n)) 
    & = \frac{1}{n \ell} \sum_{\mfu\in\mfU} \sigma_{\mfu}^2 (\bftheta)
\end{align}
where $\sigma_{\mfu}^2 (\bftheta)$ is the variance in stratum $\mfu$.
In order to compute the sample variance of the stratified sampling estimator $\tilde{f}$, it is sufficient to replace $\sigma_{\mfu}^2 (\bftheta)$ with its estimate $\hat{\sigma}_{\mfu}^2 (\bftheta,\nbar)$ computed with $\nbar$ points in  \eqref{eq:stratified_estimator_variance} (see \citealp{glasserman2004monte}).

\subsection{Variance of the Stratified Sampling Estimator}
\noindent
In this section, we establish the order of magnitude of the variance of the stratified sampling estimator \eqref{eq:stratified_estimator} in terms of a given sample size $n$.
%Thus, in our sequential sampling setting we focus on the case with fixed sample size, and carry out a comparison between order of magnitude of the variance under different specifications of the objective function.
We split the discussion into the two cases depending on whether the gradient of $\mu^{-1}$ is bounded everywhere (see also Chapter V.7 in \citealp{asmussen2007stochastic}) or not.
When not bounded everywhere (which is the case for random variables with unbounded support), we focus the analysis on distributions having mutually independent margins either with sub-exponential or with sub-Gaussian tail behavior.
At the end of the section, we show that the assumption of independence can be circumvented by suitable constructions of $\XX$, while keeping simple sampling mechanisms for $\mu$.

\begin{definition}\label{def:mu_subexp_subgaussian}
    Let $Y$ be a random variable with support $\mathcal{Y}\subseteq \RR$ and cumulative distribution function $\Psi$.
    $Y$, or equivalently, $\Psi$, is sub-exponential with parameter $\kappa_i>0$ if, for $y\in\mathcal{Y}$,
    \begin{align}\label{eq:SE}
        \mbP(|Y| \geq y) = \mcO\left(\mre^{-\kappa_i |y|}\right).
        \tag{$\SE_{\kappa_i}$}
    \end{align}
    $Y$, or equivalently, $\Psi$, is sub-Gaussian with parameter $\xi_i>0$ if, for $y\in\mathcal{Y}$,
    \begin{align}\label{eq:SG}
        \mbP(|Y| \geq y) = \mcO\left(\mre^{- \left(y/\xi_i\right)^2}\right).
        \tag{$\SG_{\xi_i}$}
    \end{align}
\end{definition}

\begin{definition}\label{def:F_poly_exp}
    $F$ is polynomial in $\XX$ with parameter $\bsa=(a_1,\ldots,a_q)\in\mathbb{N}^q$ if
    \begin{align}\label{eq:PO}
        F(\cdot,\XX) = \mcO\left( \sum_{j=1}^q X_j^{a_j} \right).
        \tag{$\PO_{\bsa}$}
    \end{align}
    $F$ is exponential in $\XX$ with parameter $\bsb=(b_1,\ldots,b_q)\in\RR^q$ if
    \begin{align}\label{eq:EX}
        F(\cdot,\XX)=\mcO\left( \mre^{\sum_j b_j X_j} \right).
        \tag{$\EX_{\bsb}$}
    \end{align}
\end{definition}

\begin{theorem}\label{prop:var_strat}
Let Assumptions \ref{ass:smooth}, \ref{ass:map}, and \ref{ass:strata} hold.
Then, given the MC sample size $n$ and the total number of strata $\ell=n/\nbar$, for $\nbar\in\mathbb{N}$, the conditional variance of the stratified sampling estimator has order of magnitude
\begin{align*}\label{eq:var_claim}
    \var \left( \tilde{f} (\bftheta_k,n)\ | \ \mcF_{k-1}\right)
    = \mcO\left( n^{-1} \ \term (n) \right),
\end{align*}
where
\begin{subnumcases}{\term (n) =}
    n^{-2/q} 
    & if \ $\| \nabla\ \mu^{-1} \| <\infty, \ \forall \uu\in[0,1]^q$, 
    \label{eq:order_bounded} 
    \\
    n^{-1/q}\left[ \log\left(n\right) \right]^{2(a_{\max}-1)}
    & if \ $\mu_i$ \ is \ref{eq:SE} $\forall i=1,\ldots,q$, and $F$ is \ref{eq:PO},
    \label{eq:order_mu_subexp_h_power}
    \\
    n^{-1/q (1-2b_{\max}/\kappa_{\min})}
    %n^{(2b_{\max}-\kappa_{\min})/(q\kappa_{\min})}
    & if \ $\mu_i$ \ is \ref{eq:SE} $\forall i=1,\ldots,q$, and $F$ is \ref{eq:EX},
    \label{eq:order_mu_subexp_h_exp}
    \\
    n^{-1/q}\left[ \log\left(n\right) \right]^{a_{\max}-1}
    & if \ $\mu_i$ \ is \ref{eq:SG} $\forall i=1,\ldots,q$, and $F$ is \ref{eq:PO},
    \label{eq:order_mu_subgaussian_h_power}
    \\
    %n^{-1/q+b_{\max} \xi_{\max} \sqrt{8/q}/\log(n)}
    %n^{-1/q} \exp\left\{ b_{\max} \xi_{\max} \sqrt{8\log(n)/q} \right\}
    n^{-1/q+b_{\max} \xi_{\max} \sqrt{8/[q\log(n)]}}
    & if \ $\mu_i$ \ is \ref{eq:SG} $\forall i=1,\ldots,q$, and $F$ is \ref{eq:EX},
    \label{eq:order_mu_subgaussian_h_exp}
\end{subnumcases}
where $a_{\max}=\max_{i}\{a_i\}$, $b_{\max}=\max_{i}\{b_i\}$, 
$\xi_{\max}=\max_{i}\{\xi_i\}$,
$\kappa_{\min}=\min_{i}\{\kappa_i\}$,
$\kappa_{\min}>2b_{\max}$,
and where in the last for cases we assume $X_i \perp X_j, \ \forall i\neq j$.

% If Assumptions \ref{ass:smooth}, \ref{ass:map}, and \ref{ass:strata} hold, the conditional variance of the stratified sampling estimator has order of magnitude
% \begin{subnumcases}{\var \left( \tilde{f} (\bftheta_k,N_k, \mfU_k)\ | \ \mcF_{k}\right)=}
%     \mcO\left(N_k^{-(q + 2)/q} \right) & if \ $\| \nabla_{\uu}\ \mu^{-1} \| <\infty, \ \forall \uu\in[0,1]^q$, 
%     \label{eq:order_bounded} 
%     \\
%     \mcO\left( N_k^{-1} h_0 \right) & if \ $X_i$ \ is SE, \ $\forall i$, \ and \ $X_i \perp X_j, \ \forall i\neq j$, 
%     \label{eq:order_subexp}
% \end{subnumcases}
% where $h_0$ depends on the expression of function $F$ ( \eqref{eq:opt}).
% In particular,
% \begin{subnumcases}{h_0=}
%     \frac{\left[ \log\left(N_k\right) \right]^{2(a_{\max}-1)}}{N_k^{1/q}} 
%     & if \ $F(\cdot,\XX) = \mcO\left( \sum_{j=1}^q X_j^{a_j} \right)$,
%     \label{eq:h0_power}
%     \\
%     N_k^{(2b_{\max}-\kappa_{\min})/q} 
%     & if \ $F(\cdot,\XX)=\mcO\left( \mre^{\sum_j b_j X_j} \right), \ \kappa_{\min}>2b_{\max}$,
%     \label{eq:h0_exp}
% \end{subnumcases}
% where $a_j\in\mathbb{N},\ b_j\in\RR, \ \forall j$, $a_{\max}=\max_{j}\{a_j\}$, $b_{\max}=\max_j\{b_j\}$, and $\kappa_{\min}=\min_j\{\kappa_j\}$.

\end{theorem}

\begin{proof}

To prove the statement, we first bound the conditional variance of a single stratum $\mfu$, i.e., 
\begin{equation*}\label{eq:var_strat}
\sigma_{\mfu}^2(\bftheta_k):= \text{Var}\left(F \left(\bftheta_k,\XX \mid \XX\in\mcX_{\mfu}\right) \ | \ \mcF_{k-1} \right),
\end{equation*}
in terms of total number of strata $\ell$.
%which is known given $\mcF_{k}$.
Let $\mfu\in\mfU$ be a strata in the hypercube according to  \eqref{eq:hypercube_strata}, and $\uu_{\mfu}\in\mfL$ be the corresponding left endpoint vector (see  \eqref{eq:hypercube_strata_leftend}).
%$s_k=T(\mathcal{X}_{s_k})$
%$\uu_i =(u_{1,s}, \ldots, u_{q,s})$ with $u_{j,s} = \min\{u : u\in \mcU_{j,s}\}$.
Let $\VV$ be a $q$-dimensional uniform random vector and 
$\alpha_{\mfu}: (0,1]^q\to \mfu$ such that $\alpha_{\mfu}(\VV):=\uu_{\mfu}+\VV/\ell^{1/q}$.
Then, using Assumptions \ref{ass:smooth} and \ref{ass:map} we apply a Taylor expansion to $F$ around $\vv\in(0,1)^q$ as
\begin{align*}
    F(\bftheta_k, \XX \mid \XX\in\mathcal{X}_{\mfu}) 
    & = F(\bftheta_k, \mu^{-1}(\UU) \mid \UU\in\mfu)
    \\
    & = F(\bftheta_k, \mu^{-1}(\alpha_{\mfu}(\VV)))
    % \\
    % & \approx F(\bftheta_k, \mu^{-1}(\alpha_{\mfu_k}(\zero))) + \sum_{j=1}^q \frac{\partial h}{\partial V_j}(\zero)V_j
    = F(\bftheta_k, \mu^{-1}(\alpha_{\mfu}(\vv))) + \frac{1}{\ell^{1/q}} \sum_{j=1}^q c_{{\mfu},j} (V_j-v_j) + \mcO\left(\frac{1}{\ell^{2/q}}\right)
\end{align*}
where
\begin{align}\label{eq:C_term}
    c_{{\mfu},j}
    & = \sum_{i=1}^q \frac{\partial F}{\partial \mu_i^{-1}} (\bftheta_k, \mu^{-1} (\alpha_{\mfu}(\vv))) \ \frac{\partial \mu_i^{-1}}{\partial \alpha_{{\mfu},j}}(\alpha_{\mfu}(\vv)).
\end{align}
Thus, for large $\ell$, the variance in each stratum reads
\begin{align*}
    \sigma^2_{\mfu} (\bftheta_k)
    & \sim \frac{1}{12 \ell^{2/q}} \sum_{j=1}^q c_{{\mfu},j}^2,
    %\label{eq:var_strat_extended}
\end{align*}
where $\sim$ denotes equality in order of magnitude.
The rest of the proof is split between the case \eqref{eq:order_bounded} and the remaining cases \eqref{eq:order_mu_subexp_h_power}-\eqref{eq:order_mu_subgaussian_h_exp}.

\begin{itemize}
\item[] \textbf{Case  
\eqref{eq:order_bounded}:} We note that the variance in each stratum can be upper-bounded as
\begin{align*}
    \sigma^2_{\mfu} (\bftheta_k)
    & \lesssim \frac{ \Cmax }{\ell^{2/q}},%\label{eq:var_strat_extended_2}
\end{align*}
where $\lesssim$ stands for \lq\lq lower or equal in order of magnitude to\rq\rq, and
where $\Cmax := \frac{1}{12} \max_{\mfu\in\mfU} \left\{ \sum_{j=1}^q c^2_{{\mfu},j} \right\}$ has order of magnitude $\mcO(1)$ 
because, by the condition in  \eqref{eq:order_bounded}, $\mu^{-1}$ is continuously differentiable over $[0,1]^q$, and $F$ is continuously differentiable too (Assumption \ref{ass:smooth}).
In order to upper-bound the conditional variance of the stratified sampling estimator, we then compute
\begin{align}
    \var(\tilde{f}(\bftheta_k, n) \ | \ \mcF_{k-1}) 
    & = \frac{1}{n \ell} \sum_{\mfu\in\mfU} \sigma_{\mfu}^2 (\bftheta_k) \lesssim n^{-1} \ell^{-2/q} \Cmax = n^{-(q+2)/q} \nbar^{2/q} \Cmax = \mcO\left( n^{-(q+2)/q} \right),
    \label{eq:var_fTilde}
\end{align}
where \eqref{eq:var_fTilde} follows from Assumption \ref{ass:strata}.

\item[] \textbf{Cases \eqref{eq:order_mu_subexp_h_power}-\eqref{eq:order_mu_subgaussian_h_exp}:}
By the postulates of the theorem, $\XX$ has mutually independent components with sub-exponential behaviors as described in Definition \ref{def:mu_subexp_subgaussian}.
In the interest of exposition, in the following we assume that $\XX$ has exponentially distributed components, i.e., $\mu_i(x)=1-\mre^{-\kappa_j x}, \ i=1,\ldots,q$, with parameters $\kappa_j>0, \ j=1,\ldots,q$.
Due to the mutually independence of $X_1,\ldots,X_q$,  \eqref{eq:C_term} reduces to
\begin{align}\label{eq:C_term_indep}
    c_{\mfu,j} = \frac{\partial F}{\partial \mu_j^{-1}} (\bftheta_k, \mu^{-1} (\alpha_{\mfu}(\vv))) \ \frac{\partial \mu_j^{-1}}{\partial \alpha_{{\mfu},j}}(\alpha_{\mfu}(\vv)).
\end{align}
If $F(\cdot,\XX)$ is linear in $\XX$, then the first term of  \eqref{eq:C_term_indep} has order $\mcO(1)$.
Define 
$\mathfrak{S} = \{ \uu_{\mfu}\ell^{1/q} \ : \ \forall\mfu\in\mfU \}$,
%$\bsmfs_{\mfu}:=\uu_{\mfu}\lbar, \ \forall\mfu\in\mfU$,
%$\bsmfs_{\mfu}=(\mfs_{\mfu,1},\ldots,\mfs_{\mfu,q}):=\uu_{\mfu} \lbar \in \{0,1,\ldots,\lbar-1\}^q:= \mfL$, 
where $\uu_{\mfu}$ is the left endpoint vector corresponding to strata $\mfu$ according to  \eqref{eq:hypercube_strata_leftend}.
Then,
\begin{align}
    \var(\tilde{f}(\bftheta_k, n) \ | \ \mcF_{k-1}) 
    & = \frac{1}{n \ell} \sum_{\mfu\in\mfU} \sigma_{\mfu}^2 (\bftheta_k)
     \sim \frac{1}{n \ell} \frac{1}{12\ell^{2/q}}
     %\sum_{\bsmfs_\mfu\in\mfL} 
     \sum_{\bsmfs_\mfu\in\mathfrak{S}}
     \sum_{j=1}^q \frac{1}{\left[\kappa_j \left(1-\frac{\mfs_{\mfu,j}+v_j}{\ell^{1/q}}\right)\right]^2}
    \label{eq:var_fTilde_subexp} \\
    & \sim \frac{1}{n \ell^{\frac{q+2}{q}}} 
    %\sum_{\bsmfs_{\mfu}\in\mfL} 
    \sum_{\bsmfs_{\mfu}\in\mathfrak{S}}
    \sum_{j=1}^q \frac{\ell^{2/q}}{(\ell^{1/q}-\mfs_{\mfu,j}-v_j)^2}
     \sim \frac{1}{n \ell} 
     %\sum_{\bsmfs_{\mfu} \in \mfL} 
     \sum_{\bsmfs_{\mfu} \in \mathfrak{S}}
     \sum_{j=1}^q \frac{1}{\mfs_{\mfu,j}^2}
     = \frac{q\ell^{\frac{q-1}{q}}}{n \ell}   \frac{\pi^2}{6}
     = \mcO\left(n^{-\frac{q+1}{q}}\right)\!,
    \nonumber
\end{align}
where $\pi\approx3.14$.

We note that the above result has analogous interpretation of the argument of Remark 7.2, Chapter V.7, in \citep{asmussen2007stochastic}.
That is, for large $\ell$, the number of strata in which the variance has order $\mcO(1)$ is of the order $\mcO(\ell^{(q-1)/q})$, which implies \eqref{eq:var_fTilde_subexp}.
We use this fact to generalize the above result for any $F$ being polynomial (see \eqref{eq:PO}). Namely, we have
\begin{align*}
    \var(\tilde{f}(\bftheta_k, n) \ | \ \mcF_{k-1}) 
    & = \frac{1}{n \ell} \mcO\left( \left[\log\left(\ell\right)\right]^{2(a_{\max}-1)} \ell^{(q-1)/q}\right)
    = \mcO\left(n^{-1} \frac{\left[\log\left(n\right)\right]^{2(a_{\max}-1)}}{n^{1/q}} \right),
\end{align*}
which proves \eqref{eq:order_mu_subexp_h_power}.
In the same fashion, we derive complexity \eqref{eq:order_mu_subexp_h_exp} in the case in which $F$ is exponential (see \eqref{eq:EX}) and $\kappa_{\min}>2b_{\max}$.
This last condition is required to ensure that the variance of $F(\cdot,\XX)$ is bounded.

If $\mu_i$ is sub-Gaussian for all $i$ (see  \eqref{eq:SG}), then $\mu_i^{-1}(u) \sim \xi_i\sqrt{2\log\left(1/(1-u)\right)}$ and $(\mu_i^{-1})'(u) \sim \frac{\xi_i}{(1-u)\sqrt{2\log(1/(1-u))}}$ as $u\to1$, while $\mu_i^{-1}(u) \sim -\xi_i\sqrt{2\log\left(1/u\right)}$ and $(\mu_i^{-1})'(u) \sim \frac{\xi_i}{u\sqrt{2\log(1/u)}}$ as $u\to0$ (see, e.g., \citealp{fung2018quantile}).
Following the above procedures, \eqref{eq:order_mu_subgaussian_h_power}-\eqref{eq:order_mu_subgaussian_h_exp} can be verified.
% for us: the sum of \mu^{-1}' in the SG case converges even faster than the \pi^2/6 observed for exponential; so results hold by using the fact that SE is worse than SG. So in practice the O() we are computing only depend on \mu^{-1} and not on \mu^{-1}', for these SG cases.
%\hfill\halmos
\end{itemize}
\end{proof}

Theorem \ref{prop:var_strat} provides a direct comparison between the variance of the stratified sampling estimator and the no-stratification estimator.
By taking the total number of strata to be $\ell=1,\forall k$, the variance of the no-stratification estimator has order $\mcO(n^{-1})$ and is always larger than the orders of magnitude in \eqref{eq:order_bounded}-\eqref{eq:order_mu_subgaussian_h_exp}.
It can be observed that especially for moderate values of $q$, the improvement of the stratified sampling estimator can be significant when $\ell$ is large.
For example, when $\|\nabla \mu^{-1}\|$ is bounded for all $\uu\in[0,1]^q$, in order for the stratified sampling estimator to yield at least a $c$-times faster decrease than the no-stratification estimator, the number of strata should be
\begin{equation}\label{eq:beta}
\ell_{k,c} \gtrsim \left\lceil c^{q/2} \right\rceil, \quad k\geq0.
\end{equation}

The next corollary highlights that we can construct factor-based random vectors with nontrivial dependence structure (see, e.g., \citealp{ballotta2016multivariate}) such that the variance of the stratified sampling estimator has an order similar to those of \eqref{eq:order_mu_subexp_h_power}-\eqref{eq:order_mu_subgaussian_h_exp}, when $\mcX$ is unbounded.
This implies that, under such specifications, $\XX$ can be conveniently sampled by only using unconditional distributions of the mutually independent factors that determine $\XX$.

\begin{corollary}\label{cor:factor_based}
    Let $\XX$ be a $\tilde{q}$-dimensional factor-based random vector of the form
    \begin{align*}
        \XX = 
        \begin{bmatrix}
            Y_1 + \sum_{m=1}^r \beta_{1,m} Z_m
            \\
            \vdots
            \\
            Y_{\tilde{q}} + \sum_{m=1}^r \beta_{\tilde{q},m} Z_m
        \end{bmatrix}.
    \end{align*}
    where $Y_1,\ldots,Y_{\tilde{q}},Z_1,\ldots,Z_r$ are mutually independent random variables and $\beta_{i,m}\in\RR, \forall i=1,\ldots,\tilde{q}, \ m=1,\ldots,r$, are coefficients.
    Then, the orders of magnitude of the variance under the four specifications of $F$ and $\mu$ of \eqref{eq:order_mu_subexp_h_power}-\eqref{eq:order_mu_subgaussian_h_exp} are the same of \eqref{eq:order_mu_subexp_h_power}-\eqref{eq:order_mu_subgaussian_h_exp} by setting $q=\tilde{q}+r$.
\end{corollary}

In the following corollary we remark that a suitable truncation of a well-behaved density function defined on $\RR^q$ can make the order of magnitude of variance be as in  \eqref{eq:order_bounded}, i.e., the lowest order of magnitude among the analyzed specifications of $F$ and $\mu$.
%\magenta{As it will be clear in the next section, the sample size should grow with the iterations in order to guarantee an almost sure convergence of the algorithm. Thus, ...}

\begin{corollary}\label{cor:truncation}
    Let $\Psi$ be a cumulative distribution function defined on an unbounded support $\mcX\subseteq\RR^q$ and $\psi$ be the corresponding probability density function.
    Let $\mu:\mcX\to\mcU$ be a Rosenblatt transformation satisfying Assumption \ref{ass:map}, where $\mcU$ is the $q$-dimensional unit hypercube open in at least one of its boundaries, and let $\|\nabla \mu^{-1}\|$ be bounded on $\mcU$.
    Let $\tilde{\Psi}$ be a cumulative distribution function with density function $\tilde{\psi}$ satisfying 
    \begin{align*}
    & \tilde{\psi}(\xx) = \frac{\psi(\xx)}{\Psi(\bsb)-\Psi(\bsa)} > 0, 
    && \text{for all $\xx\in[\bsa,\bsb]^q\subset\mcX$},
    \\
    & \tilde{\psi}(\xx) =0, 
    && \text{elsewhere},
    \end{align*}
    and let $\tilde{\mu}:[\bsa,\bsb]^q\to [0,1]^q$ be the corresponding Rosenblatt transformation.
    Then, $\|\nabla \tilde{\mu}^{-1}\|$ is bounded on $\uu\in[0,1]^q$.
\end{corollary}

\begin{proof}
    Note that $ \nabla \tilde{\mu}^{-1}(\uu) = [J_{\mu}(\xx)]^{-1} |_{\xx=\mu^{-1}(\uu)}$,
    where $J_{\mu}(\xx) := \textrm{diag} (\psi_1(x_1), \ldots,\psi_q(x_q|x_{1:q-1}))$.
    Since $\psi(\xx)>0, \ \forall \xx$, then we also have $\psi_1(x_1), \ldots,\psi_q(x_q|x_{1:q-1})>0, \ \forall \xx$.
    Thus, $\|\nabla \tilde{\mu}^{-1}\|$ is bounded on any $\uu\in[0,1]^q$.
%\hfill\halmos
\end{proof}

\subsection{Convergence}
\noindent
In this section we establish the convergence of the SASTRO-DF algorithm.
To this purpose, we analyze the limiting behavior of the error between the stratified sampling estimator and the objective function, defined as
\begin{equation}\label{eq:error_OF}
    E\left(\bftheta_k^{(i)}, N_k^{(i)}\right) := \tilde{f} \left(\bftheta_k^{(i)},N_k^{(i)}\right) - f(\bftheta_k^{(i)}),
    \quad k\geq0,
    \quad i=0,1,\ldots,p+1.
\end{equation}
We show that by a suitable choice of the adaptive sampling rule, the error is bounded by a term $\mcO(\Delta_k^2)$ for large $k$ almost surely; this will be a core ingredient to show the almost sure convergence of the algorithm as shown later.

%in such a way that the trust region model is \emph{stochastically fully linear} (in the sense of Definition 2.3 \citealp{ha2025iteration}) with respect to $f$, which in turn assures the convergence of our algorithm.

%Recall from  \eqref{eq:var_fTilde} that the 

The general expression that we adopt for the adaptive sampling rule is, for a given trust-region radius $\Delta_k$,
\begin{equation}\label{eq:adaptive_sampling_rule}
   N_k^{(i)}=\min\left\{ n\in\mcN_q, \ n\geq\lambda_k \ : \ \sqrt{\varHat_0(\tilde{f}(\bftheta_k^{(i)},n) \mid \mcF_{k-1})} \leq \frac{\kappa_{as} \Delta_k^{\gamma}}{\sqrt{\lambda_k}}
   \right\},
   \quad k\geq0,
   \quad i=0,1,\ldots,p+1,
\end{equation}
where
\begin{align}\label{eq:varHat_fTilde}
    \varHat_0\left(\tilde{f}(\bftheta_k^{(i)}, n)\mid \mcF_{k-1}\right) 
    & = \frac{1}{n} \max\left(\sigmamin^2, \frac{1}{l}\sum_{\mfu\in\mfU} \hat{\sigma}_{\mfu}^2 \left(\bftheta_k^{(i)}, \nbar\right) \right),
\end{align}
where $\hat{\sigma}^2_{\mfu} (\bftheta_k^{(i)}, \nbar)$ is the sample variance of $F(\bftheta_k^{(i)}, \XX \mid \XX\in\mfu)$ conditioned to $\mcF_{k-1}$ and estimated with $\nbar$ points, $\sigmamin^2$ is an arbitrary small constant acting as a lower bound of the sample variance (to avoid early stopping due to a probable largely underestimated $\sigma_{\mfu}^2$), $\mcN_q$ is defined in  \eqref{eq:set_N}, and $\kappa_{as}>0$ is an arbitrary constant.
Note that for $\sigmamin^2=0$,  \eqref{eq:varHat_fTilde} is simply the variance of the stratified sampling estimator (see Chapter 4.3 in \citealp{glasserman2004monte}) under Assumption \ref{ass:strata}.

For ease of exposition, in this section we focus on the sample sizes $N_k:=N_k^{(0)},k\geq0,$ associated with the center points of the trust regions, but the statements remain valid even when the sample sizes of the other interpolation and candidate points are considered.

The next theorem establishes the (optimal) values of $\lambda_k$ and $\gamma$ in  \eqref{eq:adaptive_sampling_rule} that yield that the stochastic error is upper-bounded by a term of order $\mcO(\Delta_k^2)$ for large $k$ almost surely.
In the interest of generality, in the following we only rely on Assumption \ref{ass:smooth} (bounded variance of $F(\cdot,\XX)$), which is sufficient to bound the stochastic error through Chebyshev inequality (see also Lemma 5.1 in \citealp{shashaani2018astro} for a similar application).
This marks a difference with respect to the results of \citep{ha2025complexity}, where $F(\cdot,\XX)$ is assumed to have a sub-exponential tail behavior and the associated bound is computed by means of Bernstein-type inequalities.

\begin{theorem}\label{prop:stochastic_error}
Let Assumptions \ref{ass:smooth} and \ref{ass:strata} hold and let the adaptive sampling rule be set as in  \eqref{eq:adaptive_sampling_rule}.
Let $\XX$ and $F$ be as in the optimization problem \eqref{eq:opt}.
Let $\|\nabla \mu^{-1}\|<\infty, \ \forall \uu\in(0,1)^q$.
Set
\begin{subnumcases}{}
    \lambda_k = k^{\frac{(1+\delta)q}{q+2}},
    \quad \gamma = \frac{2q}{q+2}, 
    & \text{if \ $\|\nabla\ \mu^{-1}\|<\infty, \ \forall \uu\in[0,1]^q$}
    \label{eq:lambda_gamma}
    \\
    \lambda_k = k^{\frac{(1+\delta)q}{q+1}},
    \: \gamma = \frac{(4+\varrho)q}{2(q+1)}, 
    & \text{if \ $\mu_i$ \ is \ref{eq:SE} $\forall i=1,\ldots,q$, and $F$ is \ref{eq:PO}}
    \label{eq:lambda_gamma_mu_subexp_h_power}
    \\
    \lambda_k = k^{\frac{(1+\delta)q\kappa_{\min}}{(q+1)\kappa_{\min}-2b_{\max}}},
    \quad \gamma = \frac{2q\kappa_{\min}}{(q+1)\kappa_{\min}-2b_{\max}}, 
    & \text{if \ $\mu_i$ \ is \ref{eq:SE} $\forall i=1,\ldots,q$, and $F$ is \ref{eq:EX}}
    \label{eq:lambda_gamma_mu_subexp_h_exp}
    \\
    \lambda_k = k^{\frac{(1+\delta)q}{q+1}},
    \: \gamma = \frac{(4+\varrho)q}{2(q+1)}, 
    & \text{if \ $\mu_i$ \ is \ref{eq:SG} $\forall i=1,\ldots,q$, and $F$ is \ref{eq:PO}}
    \label{eq:lambda_gamma_mu_subgaussian_h_power}
    \\
    \lambda_k = k^{\frac{(1+\delta)q}{q+1}},
    \: \gamma = \frac{(4+\varrho)q}{2(q+1)},
    & \text{if \ $\mu_i$ \ is \ref{eq:SG} $\forall i=1,\ldots,q$, and $F$ is \ref{eq:EX}}
    \label{eq:lambda_gamma_mu_subgaussian_h_exp}
\end{subnumcases}
for all $\delta, \varrho > 0$, and where $\kappa_{\min}=\min_{i}\{\kappa_i\}$,
$b_{\max}=\max_{i}\{b_i\}$, $\kappa_{\min}>2b_{\max}$,
with $\kappa_i, b_i, i=1,\ldots,q,$ being described in Definitions \ref{def:mu_subexp_subgaussian}-\ref{def:F_poly_exp}.
Assume also that in cases \eqref{eq:lambda_gamma_mu_subexp_h_power}-\eqref{eq:lambda_gamma_mu_subgaussian_h_exp} it holds that $X_i\perp X_j, \forall i\neq j$.
%, and that \magenta{$\Delta_k\leq\mcO(1)$, for $k\to\infty$}.
Then, it holds that
% \begin{equation}\label{eq:borel_cantelli}
% \mbP\{ |E_k(\bftheta_k, N_k)| \geq \kappa_{fde} \Delta_k^2 \quad  \text{i.o.} \} = 0.
% \end{equation}
% KEEP
\begin{equation}\label{eq:borel_cantelli}
    \mbP\left\{ \limsup_{k\to\infty} | E(\bftheta_k,N_k) | \geq \kappa_{fde} \Delta_k^2 \right\} = 0,
    \quad \forall \kappa_{fde}>0.
\end{equation}

%\giovanni{}
\end{theorem}

\begin{proof}
As usual we split the proof into the case in which $\XX$ has bounded support \eqref{eq:lambda_gamma} and the case of unbounded support \eqref{eq:lambda_gamma_mu_subexp_h_power}-\eqref{eq:lambda_gamma_mu_subgaussian_h_exp}.
\begin{itemize}
\item[]
\textbf{Case \eqref{eq:lambda_gamma}:} By using the law of total probability and Chebyshev's inequality, for any $\kappa_{fde}>0$ we have
\begin{align*}
    \mbP\{ |E(\bftheta_k, N_k)| \geq \kappa_{fde} \Delta_k^2 \mid \mcF_{k-1}\}
    \leq \frac{\mbE[E^2(\bftheta_k, N_k)\mid \mcF_{k-1}]}{\kappa_{fde}^{2} \Delta_k^{4}} 
    = \frac{\text{Var}(\tilde{f}(\bftheta_k,N_k)\mid \mcF_{k-1})}{\kappa_{fde}^2 \Delta^4_k},
\end{align*}
where the last equivalence holds because the estimator $\tilde{f}(\bftheta_k,N_k)$ is assumed to be unbiased.
Given the order of magnitude of the variance of  \eqref{eq:order_bounded} and the sampling rule \eqref{eq:adaptive_sampling_rule}, we have that
\begin{align}
    \frac{\text{Var}(\tilde{f}(\bftheta_k,N_k)\mid \mcF_{k-1})}{\kappa_{fde}^2 \Delta^4_k} 
    & \lesssim \frac{N_k^{-(q+2)/q}\: \nbar^{2/q} \Cmax}{\kappa_{fde}^2 \Delta_k^4}.
    %\label{eq:first}
    \notag
    \\
    & \leq \frac{\nbar^{2/q} \Cmax\ \kappa_{as}^{2(q+2)/q} \sigmamin^{-2(q+2)/q}}{\kappa_{fde}^2} \lambda_k^{-(q+2)/q},
    \label{eq:second}
\end{align}
where we have selected $\gamma=2q/(q+2)$ in such a way that the trust-region radius is removed from the equation.
Note that, following Theorem 1 in \citep{ghosh1979sequential}, the variance of the estimator with a given sample size \eqref{eq:var_fTilde} has the same order of magnitude of the variance of the estimator under a sequential sampling scheme (see also Theorems 2.6-2.7 in \citealp{shashaani2018astro}).

In order to prove \eqref{eq:borel_cantelli}, by Borel-Cantelli lemma it is sufficient to prove that the right-hand side of  \eqref{eq:second} is summable in $k$.
Following the choice of $\lambda_k$ in  \eqref{eq:lambda_gamma}, it holds that
\begin{align*}
\sum_{k=1}^{\infty} \lambda_k^{-(q+2)/q}
& \leq \sum_{k=1}^{\infty} k^{-(1+\delta)} < \infty,
\quad \forall\delta>0.
%\label{eq:summable_1}
\end{align*}

\item[]
\textbf{Cases \eqref{eq:lambda_gamma_mu_subexp_h_power}-\eqref{eq:lambda_gamma_mu_subgaussian_h_exp}:} \eqref{eq:lambda_gamma_mu_subexp_h_exp} is proven similarly to the previous point.
Note that
\begin{align*}
    \frac{\text{Var}(\tilde{f}(\bftheta_k,N_k)\mid \mcF_{k-1})}{\kappa_{fde}^2 \Delta^4_k}
    & \lesssim \frac{1}{\Delta_k^4} (\lambda_k \Delta^{-2\gamma})^{-((q+1)\kappa_{\min}-2b_{\max})/(q\kappa_{\min})}.
\end{align*}
By setting $\lambda_k,\gamma$ as in \eqref{eq:lambda_gamma_mu_subexp_h_exp}, it can be observed that terms depending on $\Delta_k$ cancel out and we only remain with a summable term $k^{-(1+\delta)}, \ \delta>0$.

To prove  \eqref{eq:lambda_gamma_mu_subexp_h_power}, we note that, for $a_{\max}\in\mathbb{N}$,
\begin{align*}
    \frac{\text{Var}(\tilde{f}(\bftheta_k,N_k)\mid \mcF_{k-1})}{\kappa_{fde}^2 \Delta^4_k}
    & \lesssim \frac{N_k^{-(q+1)/q} [\log(N_k)]^{2(a_{\max}-1)}}{\Delta_k^4}
    \\
    & \sim \frac{1}{\Delta_k^4} (\lambda_k \Delta_k^{-2\gamma})^{-(q+1)/q} \ [\log(\lambda_k \Delta_k^{-1})]^{2(a_{\max}-1)}.
\end{align*}
Set $\lambda_k, \gamma$ as in  \eqref{eq:lambda_gamma_mu_subexp_h_power}, then
\begin{align*}
    \mbP\{ |E(\bftheta_k, N_k)| \geq \kappa_{fde} \Delta_k^2 \mid \mcF_{k-1}\}
    & \lesssim k^{-(1+\delta)} \Delta_k^{\varrho}\ \left[ \log(k) + \log(\Delta_k^{-1}) \right]^{2(a_{\max}-1)}.
\end{align*}
The right-hand side is summable in $k$ if (i) $k^{-(1+\delta)} \log(k)^{a}$ is summable for all $a\in\mathbb{N}, \ \delta>0$,
and (ii) $(\Delta_k^{-1})^{-\varrho}\ [\log(\Delta_k^{-1})]^{a} \leq \mcO(1)$ for $k\to\infty$ and for all $a\in\mathbb{N}, \ \varrho>0$.
(i) is easily verified and (ii) holds because $\Delta_k \leq \Deltamax= \mcO(1), \forall k$.
 \eqref{eq:lambda_gamma_mu_subgaussian_h_power} is proven with the same considerations.

To prove  \eqref{eq:lambda_gamma_mu_subgaussian_h_exp}, by setting $\lambda_k$ and $\gamma$ as in the equation we note that
\begin{align}
    \frac{\text{Var}(\tilde{f}(\bftheta_k,N_k)\mid \mcF_{k-1})}{\kappa_{fde}^2 \Delta^4_k}
    & \lesssim
    \frac{
    k^{-(1+\delta)} \Delta_k^{\varrho}
    \exp\left\{ b_{\max} \xi_{\max} \sqrt{8/q} \sqrt{\log\left( k^{\frac{(1+\delta)q}{q+1}} \Delta_k^{-\frac{(4+\varrho)q}{2(q+1)}} \right)} \right\}
    }{
    \log\left(k^{\frac{(1+\delta)q}{q+1}} \Delta_k^{-\frac{(4+\varrho)q}{2(q+1)}} \right)
    }
    \notag
    \\
    & \lesssim
    \frac{
    k^{-(1+\delta)} \Delta_k^{\varrho}
    \exp \left\{\hat{a}\sqrt{\log(k)} \right\}
    \exp \left\{\hat{b}\sqrt{\log(\Delta_k^{-1})} \right\}
    }{
    \log(k) + \log(\Delta_k^{-1})
    }
    \quad \text{for some $\hat{a},\hat{b}>0$}
    \notag
    \\
    & =
    \frac{
    k^{ -(1+\delta)+ \hat{a}/\sqrt{\log(k)}}
    }{
    \log(k) + \log(\Delta_k^{-1})
    }
    \: \Delta_k^{ \varrho + \hat{b}/\sqrt{\log(1/\Delta_k)}}.
    \label{eq:here}
\end{align}
In the left term, the numerator is summable because for all $k>\mre^{(\hat{a}/\delta)^2}$, we get an infinite sum of the form $\sum_{k=k_0}^{\infty} k^{-\tilde{a}}, \ \tilde{a}>1$.
As $\Delta_k\leq \mcO(1), \forall k$,
the denominator of the left term in \eqref{eq:here} does not affect the summability of the whole fraction.
The right term in \eqref{eq:here} has order lower or equal to $\mcO(1)$ because the exponent is strictly positive and $\Delta_k\leq \mcO(1), \forall k$; thus, the whole term is summable.
\end{itemize}
%\hfill\halmos
\end{proof}

Note that the choices of $\lambda_k$ and $\gamma$ could have been formulated as inequalities. For example, under the assumptions that led to  \eqref{eq:lambda_gamma}, we may set
\begin{align}
    & \lambda_k \geq k^{\frac{(1+\delta)q}{q+2}},
    \quad \gamma \geq \frac{2q}{q+2},
    \quad\forall\delta>0,
\end{align}
without affecting the results of Theorem \ref{prop:stochastic_error}.
In fact,  \eqref{eq:lambda_gamma} only represent the values of $\lambda_k$ and $\gamma$ that make the total number of samples more parsimonious, at the expense of a possibly larger number of iterations.
However, the regulation of these two opposite forces can be carried out, e.g., by a suitable choice of the scaling constant $\kappa_{as}$ reported in  \eqref{eq:adaptive_sampling_rule}.

Note also that following Theorem \ref{prop:stochastic_error}, it can be easily verified that under a no-stratification strategy, in order to ensure almost sure convergence we would need to set (at least)
\begin{align}\label{eq:lambda_gamma_NS}
    & \lambda_k^{\textrm{NS}} = k^{1+\delta},
    \quad \gamma^{\textrm{NS}} = 2,
    \quad \forall\delta>0.
\end{align}
These values are greater than the values of $\lambda_k$ and $\gamma$ under all specifications of $F$ and $\mu$ as reported in \eqref{eq:lambda_gamma}-\eqref{eq:lambda_gamma_mu_subgaussian_h_exp}, which, as detailed in the next section, allows our dynamic stratification to require less number of function evaluations to reach a given stationary point with respect to a pure random sampling strategy.

\begin{lemma}\label{lem:successful_iteration}
(Lemma 4.4, \citealp{ha2025complexity})
    There exists a finite random iteration $K_{\rho}$ such that, for all $k\geq K_{\rho}$, the inequality $\Delta_k\leq \kappa_{dum} \| \nabla M_k \|$, for finite $\kappa_{dum}>0$, implies a successful iteration almost surely.
    Thus,
    \begin{align}\label{eq:DeltaGreaterThanG}
        \Delta_k \geq \mcO(\|\nabla M_k\|),
        \quad \forall k\geq K_{\rho},
        \quad \text{a.s.}
    \end{align}
\end{lemma}

The next theorem shows that the SASTRO-DF converges almost surely to a first order stationary point, as a consequence of Theorem \ref{prop:stochastic_error} and Lemma \ref{lem:successful_iteration}.

% \begin{corollary}\label{cor:convergence_of_the_algo}
% Let $M_k: \mcB(\bftheta_k,\Delta)\to \mbR$ be a stochastic trust region model for the approximation of $\tilde{f}$ in the ball $\mcB(\boldsymbol{\bftheta_k},\Delta)=\{\bftheta\in\mbR^q: \| \bftheta-\bftheta_k \|_2\leq\Delta\}$ at iteration $k$.
% Let the true objective function $f$ have Lipschitz continuous gradient (Assumption \ref{ass:lipschitz}).
% Let $M_k$ be a fully linear model of the estimated objective function $\tilde{f}$ (see Definition 10.3 in \citealp{conn2009introduction}).
% Then $M_k$ is a stochastic fully linear model of $f$ in the trust region, that is,
% \begin{align}
%     & |f(\bftheta) - M_k(\bftheta)| \leq \kappa_f \Delta_k^2
%     \label{eq:SFL1}
%     \\
%     & \|\nabla f(\bftheta) - \nabla M_k(\bftheta)\| \leq \kappa_g \Delta_k
%     \label{eq:SFL2}
% \end{align}
% almost surely for $k\to\infty$, for some $\kappa_f,\kappa_g>0$, and for all $\bftheta\in\mathcal{B}(\bftheta_k,\Delta)$.
% It follows that
% \begin{equation}\label{eq:convergence}
%     \|\nabla f (\bftheta_k) \| \to 0,
%     \quad \text{for} \quad k\to\infty \quad \text{a.s.}
% \end{equation}
% \end{corollary}

% \begin{proof}
%     The stochastic full linearity follows from Theorem \ref{prop:stochastic_error} ( \eqref{eq:borel_cantelli}) and the almost sure convergence follows from Theorem 4.5 in \citealp{ha2025iteration}. 
%     \magenta{[check]}
% \hfill\halmos
% \end{proof}

%%%%
\begin{theorem}\label{cor:convergence_of_the_algo}
Let $M_k,\ k\geq0,$ be the trust-region models as constructed in Definition \ref{def:polynomial_model_DQ}, and let their Hessians $\mathsf{B}_k,\ k\geq0,$ have bounded norm $\| \mathsf{B}_k \| \leq \kappa_H < \infty, \ \forall k\geq0$.
Then, the SASTRO-DF algorithm converges almost surely to a first order stationary point, i.e.
\begin{align}
    \|\nabla f (\bftheta_k) \| \to 0,
    \quad \text{for} \quad k\to\infty \quad \text{a.s.}
    \notag
\end{align}
\end{theorem}

\begin{proof}
By triangle inequality, we have
\begin{align}\label{eq:convergence_triang}
    \| \nabla f(\bftheta_k)\| \leq \| \nabla f(\bftheta_k) - \nabla M_k(\bftheta_k)\| + \| \nabla M_k(\bftheta_k)\|,
    \quad k\geq 0,
\end{align}
so we can prove the statement by showing that the right-hand side vanishes for $k\to\infty$ almost surely.
Following Lemma 2.8(ii) in \citep{shashaani2018astro}, it holds that $\| \nabla f(\bftheta_k) - \nabla M_k(\bftheta_k)\| \leq \mcO(\Delta_k)$ almost surely for sufficiently large $k$, for any linear trust-region model; similarly, this holds for the DQ-SPIM model $M_k$ of Definition \ref{def:polynomial_model_DQ} (see also Chapter 2-3 in \citep{conn2009introduction} for details on the procedure).
Then, it is sufficient to show that $\Delta_k \overset{\text{a.s.}}{\underset{k\to\infty}{\longrightarrow}} 0$ to prove that the right-hand side in  \eqref{eq:convergence_triang} goes to zero, where $\|\nabla M_k\|$ is also upper bounded by $\mcO(\Delta_k)$ by Lemma \ref{lem:successful_iteration}.
%where $\|\nabla M_k\|$ is also upper bounded by $\mcO(\Delta_k)$ as a consequence of the acceptance rule \eqref{eq:trust_region_update} \magenta{(see also \citealp{ha2025complexity})}.

To prove that $\Delta_k\overset{\text{a.s.}}{\underset{k\to\infty}{\longrightarrow}}0$, following Lemma 4.3 in \citep{ha2025complexity} we show that $\Delta_{k_i}^2$ is summable in $k_i$, where $\mcA:=\{k_1,k_2,\ldots\}$ is the set of successful iterations (for unsuccessful iterations, the trust-region radius vanishes by means of rule \eqref{eq:trust_region_update}).
We first note that
\begin{align}
    \infty 
    & > f(\bftheta_{k_1})-f(\bftheta^*)
    \label{eq:term_first}
    \\
    & \geq 
    \sum_{i=1}^{\infty} [\tilde{f}(\bftheta_{k_i}, N_{k_i}) - \tilde{f}(\bftheta_{k_{i+1}},N_{k_{i+1}})] 
    \label{eq:term_estimate}
    \\
    & \quad + 
    \sum_{i=1}^{\infty} [E(\bftheta_{k_i}, N_{k_i}) - E(\bftheta_{k_{i+1}},N_{k_{i+1}})],
    \label{eq:term_error}
\end{align}
where, due to $\nabla f$ being Lipschitz continuous (Assumption \ref{ass:lipschitz}), $f(\bftheta_{k_1})$ is finite whenever $\bftheta_{k_1}$ is finite, and $\bftheta_{k_1}$ is finite in any trust region.

Now, consider the first sum in  \eqref{eq:term_estimate}.
%Let $\mathsf{B}_k$ be the Hessian of $M_k$ constructed as in Definition \ref{def:polynomial_model_DQ}, and let its norm be bounded as $\| \mathsf{B}_k \| \leq \kappa_H < \infty, \ \forall k\geq0$, as postulated in the theorem.
Standard computations involving Cauchy steps and the trust region update rule \eqref{eq:trust_region_update} yield
\begin{align*}
    \tilde{f}(\bftheta_{k_i}, N_{k_i}) - \tilde{f}(\bftheta_{k_{i+1}},N_{k_{i+1}})
    & \geq \eta \ [M_{k_i}(\bftheta_{k_i}) - M_{k_{i+1}}(\bftheta_{k_{i+1}})]
     \geq \kappa_c \Delta_{k_i}^2, 
    \qquad \kappa_c:= \frac{\eta}{2\tilde{\eta}} \left( \frac{1}{\tilde{\eta}\kappa_H} \wedge 1 \right),
\end{align*}
for $\eta,\tilde{\eta}>0$. 
Using this result, we can rearrange \eqref{eq:term_first}-\eqref{eq:term_error} as
\begin{align}
    f(\bftheta^*)-f(\bftheta_{k_1}) + \kappa_c \sum_{i=1}^{\infty} \Delta_{k_i}^2 
    & \leq \sum_{i=1}^{\infty} [E(\bftheta_{k_i}, N_{k_i}) - E(\bftheta_{k_{i+1}},N_{k_{i+1}})]
    \leq 2 \sum_{i=1}^{\infty} | E(\bftheta_{k_i}, N_{k_i}) |
    \label{eq:2_times_error}
\end{align}

Following Theorem \ref{prop:stochastic_error}, let now $K_{fde}$ be the iteration such that $| E(\bftheta_{k}, N_{k}) | < \kappa_{fde}\Delta_{k}^2, \ \forall k\geq K_{fde}$, for an arbitrary $\kappa_{fde}>0$, and let $I_{fde}$ be the first successful iteration starting from $K_{fde}$.
Then, it holds that
\begin{align*}
    \sum_{i=1}^{\infty} | E(\bftheta_{k_i}, N_{k_i}) | \leq \sum_{i=1}^{I_{fde}-1} | E(\bftheta_{k_i}, N_{k_i}) |
    + \kappa_{fde} \sum_{i=I_{fde}}^{\infty} \Delta_{k_i}^2,
    \qquad \text{a.s.}
\end{align*}
Using this result and  \eqref{eq:2_times_error} in \eqref{eq:term_first}-\eqref{eq:term_error}, we get
\begin{align*}
    \sum_{i=I_{fde}}^{\infty} \Delta_{k_i}^2
    \leq \frac{1}{\kappa_c - 2\kappa_{fde}} \left(
    f(\bftheta_{k_1} - f(\bftheta^*) + 2\sum_{i=1}^{I_{fde}-1} | E(\bftheta_{k_i}, N_{k_i}) |
    \right),
    \qquad \text{a.s.},
\end{align*}
where the right-hand side is positive and finite for any $0<\kappa_{fde}<\kappa_c/2$, ensuring that $\Delta_k\overset{\text{a.s.}}{\underset{k\to\infty}{\longrightarrow}}0$.
%\hfill\halmos
\end{proof}

%%% KEEP
% \magenta{
% \begin{lemma}\label{lemma:Delta}
% \giovanni{NEEDED?}
%     The trust-region radius $\Delta_k$ is such that
%     \begin{equation}\label{eq:Delta2}
%         \sum_{k=0}^{\infty} \Delta^2_k<\infty,
%     \end{equation}    \begin{equation}\label{eq:Delta}
%         \lim_{k\to\infty}\Delta_k=0.
%     \end{equation}
% \end{lemma}
% \begin{proof}
%     The result follows from Lemma 4.3 in \citep{ha2025complexity}.
%     \magenta{[need a check]}
%     %KEEP
%     %also due to trust-region update rule of \eqref{property:model_agreement}, \eqref{eq:trust_region_update},
% \end{proof}
% }

\subsection{Complexity}\label{sec:complexity}
\noindent
In this section we establish the sample complexity of the SASTRO-DF algorithm. 
The sample complexity measures the number of oracle calls $F(\cdot)$ the algorithm needs across all iterations to achieve $\epsilon$-accuracy.
We denote the total oracle calls after $k$ iterations as $W_k$ (see \eqref{eq:total_samples_PRELIMINARIES}).
% \begin{equation}\label{eq:total_samples}
%     W_{k} = \sum_{l=0}^k \left(\sum_{i=0}^p N_l^{(i)}+N_l^{s} \right),
% \end{equation}
% where $p$ is the number of noncentral points used to construct the polynomial approximation model in $\mcB(\bftheta_k,\Delta_k), \forall k$, which we set as $p=2d$ following the DQ-SPIM construction of Definition \ref{def:polynomial_model_DQ}.
Then, we say that the algorithm exhibits $\mcO(\epsilon^{-m})$ sample complexity if 
\begin{align}
    & W_{\epsilon} 
    = \sum_{k=0}^{T_\epsilon} \left( \sum_{i=0}^p N_k^{(i)} + N_k^s \right) = \mcO(\epsilon^{-m}),
    \qquad \text{a.s.}
    \label{eq:sample_complexity_def} \\
    & \text{for\quad} T_\epsilon := \inf\{k: \| \nabla f (\bftheta_k)\|_2 \leq \epsilon\},
    \label{eq:T_epsilon}
\end{align}
and for a \emph{sufficiently small} $\epsilon$, to be specified later in the section.

In order to establish the sample complexity it is first necessary to identify the iteration complexity, which is addressed by the following corollary.

\begin{lemma}\label{lemma:iteration_complexity}
% KEEP
% If Assumptions \ref{ass:lipschitz}, \ref{ass:smooth}, \ref{ass:strata} hold, the trust region models are stochastically fully linear, and the trust region is updated according to \eqref{property:model_agreement} and \eqref{eq:trust_region_update} \giovanni{are all these necessary and sufficient conditions?}, then 
The number of iterations required to reach an $\epsilon$-first-order stationary point, i.e., $\|\nabla f(\bftheta_k)\|\leq\epsilon$, satisfies 
%\sara{state probabilistically}
\begin{align}\label{eq:iter_complexity}
T_\epsilon = \mathcal{O} \left( \epsilon^{-2} \right),
\qquad \text{a.s.},
\end{align}
for a sufficiently small $\epsilon>0$.
%
%That is, the SASTRO-DF algorithm yields $\mcO(\epsilon^{-2})$ iteration complexity. 
\end{lemma}

\begin{proof}
The proof follows analogous steps of Theorem 5.1 in \citep{ha2025complexity} used for the ASTRO-DF algorithm.
Under the stratification framework, the results of Theorem \ref{prop:stochastic_error} and Corollary \ref{cor:convergence_of_the_algo} ensure we obtain the same iteration complexity, as summarized in the following.

First, note that by Lemma 2.8(ii) in \citep{shashaani2018astro}, we have that for $\kappa_{mge}>0$, there exists a finite iteration $K_{mge}$ such that $\| \nabla f(\bftheta_k) - \nabla M_k(\bftheta_k)\| \leq \kappa_{mge}\Delta_k$ almost surely, for all $k\geq K_{mge}$.
In addition, by Lemma \ref{lem:successful_iteration}, we know that $\Delta_k \geq \mcO(\|\nabla M_k\|), \ k\geq K_{\rho}$, for a finite iteration $K_{\rho}$.
% In addition, by Lemma 4.4 in \citep{ha2025complexity} we know that for any $\kappa_{dum}>0$, there exists a finite iteration $K_{\rho}$ such that, for all $k\geq K_{\rho}$, the inequality $\Delta_k\leq \kappa_{dum} \| \nabla M_k \|$ implies a successful iteration almost surely, so that
% \begin{align}\label{eq:DeltaGreaterThanG}
%     \Delta_k \geq \mcO(\|\nabla M_k\|).
% \end{align}

Define now the iteration 
\begin{align}\label{eq:Krhomge}
    K_{\rho,mge}=\max\{K_{\rho},K_{mge}\},
\end{align}
and set an $\epsilon_0>0$ sufficiently small so that $K_{\rho,mge}<T_{\epsilon_0}$, where $T_{\epsilon_0}$ is defined in \eqref{eq:T_epsilon}. 
For all $\epsilon\leq\epsilon_0$,
By using the relation \eqref{eq:DeltaGreaterThanG} and the definition of $T_\epsilon$ in \eqref{eq:T_epsilon},
we have that $\Delta_k\geq\kappa_{l\Delta} \epsilon$ for some $\kappa_{l\Delta}>0$ and for $K_{\rho,mge}\leq k < T_{\epsilon}$ (see also Lemma 4.5 in \citealp{ha2025complexity})
Then, by using the result that $\sum_{k=0}^{\infty} \Delta_k^2$ is finite almost surely (Corollary \ref{cor:convergence_of_the_algo}), we have
\begin{align}
    \infty 
    & > \sum_{k=K_{\rho,mge}}^{T_{\epsilon}} \hspace{-10pt} \Delta_k^2
    \notag
    \\
    & \geq (T_\epsilon-K_{\rho,mge}) \kappa_{l\Delta}^2 \epsilon^2.
    \notag
\end{align}
Thus, $T_{\epsilon}\epsilon^2$ is bounded almost surely, implying \eqref{eq:iter_complexity}.
%\hfill\halmos
\end{proof}

The next theorem establishes the sample complexity of the SASTRO-DF algorithm.

\begin{theorem}\label{prop:sample_complexity}
The number of oracle calls required to reach an $\epsilon$-first-order stationary point, i.e., $\nabla f(\bftheta_k)\leq\epsilon$, satisfies, for sufficiently small $\epsilon>0$,
\begin{subnumcases}{W_{\epsilon}=}
    \mcO \left( \epsilon^{-\frac{8q+4+2q\delta}{q+2}} \right), 
    & \text{if \ $\|\nabla\ \mu^{-1}\|<\infty, \ \forall \uu\in[0,1]^q$}
    \label{eq:sample_complexity}
    \\
    \mcO \left( \epsilon^{-\frac{8q+2+(2\delta+\varrho)q}{q+1}} \right), 
    & \text{if \ $\mu_i$ \ is \ref{eq:SE} $\forall i=1,\ldots,q$, and $F$ is \ref{eq:PO}}
    \label{eq:sample_complexity_mu_subexp_h_power}
    \\
    \mcO \left( \epsilon^{-\frac{
    (8q+2\delta q +2)\kappa_{\min}-4b_{\max}
    }{
    (q+1)\kappa_{\min}-2b_{\max}}
    } \right), 
    & \text{if \ $\mu_i$ \ is \ref{eq:SE} $\forall i=1,\ldots,q$, and $F$ is \ref{eq:EX}}
    \label{eq:sample_complexity_mu_subexp_h_exp}
    \\
    \mcO \left( \epsilon^{-\frac{8q+2+(2\delta+\varrho)q}{q+1}} \right), 
    & \text{if \ $\mu_i$ \ is \ref{eq:SG} $\forall i=1,\ldots,q$, and $F$ is \ref{eq:PO}}
    \label{eq:sample_complexity_mu_subgaussian_h_power}
    \\
    \mcO \left( \epsilon^{-\frac{8q+2+(2\delta+\varrho)q}{q+1}} \right),
    & \text{if \ $\mu_i$ \ is \ref{eq:SG} $\forall i=1,\ldots,q$, and $F$ is \ref{eq:EX}}
    \label{eq:sample_complexity_mu_subgaussian_h_exp}
\end{subnumcases}
almost surely, for arbitrarily small for all $\delta, \varrho > 0$, where $\kappa_{\min}=\min_{i}\{\kappa_i\}$,
$b_{\max}=\max_{i}\{b_i\}$, $\kappa_{\min}>2b_{\max}$,
with $\kappa_i, b_i, i=1,\ldots,q,$ being described in Definitions \ref{def:mu_subexp_subgaussian}-\ref{def:F_poly_exp},
and in cases \eqref{eq:sample_complexity_mu_subexp_h_power}-\eqref{eq:sample_complexity_mu_subgaussian_h_exp} we have assumed that $X_i\perp X_j, \forall i\neq j$.
\end{theorem}

% \begin{align}\label{eq:sample_complexity_SASTRO}
%     W_{\epsilon} = \mcO \left( \epsilon^{-\frac{8q+4+2q\delta}{q+2}} \right),
%     \qquad \text{a.s.},
%     \qquad \text{if \ $\|\nabla\ \mu^{-1}\|<\infty, \ \forall \uu\in[0,1]^q$}
% \end{align}
%That is, The SASTRO-DF algorithm has $\mcO \left( \epsilon^{-\frac{8q+4+2q\delta}{q+2}} \right)$ sample complexity.

\begin{proof}
When $\XX$ has unbounded support, the sample variance converges almost surely to the true variance as the sample size grows (see, e.g., Theorem 2.6 in \citealp{shashaani2018astro}).
Thus, by using Borel-Cantelli lemma it can be observed that for any $\kappa_{\sigma}>1$,
\begin{align}
    \mbP\left\{
    \limsup_{k\to\infty}
    \frac{1}{L_k}\sum_{\mfu\in\mfU_k}\hat{\sigma}_{\mfu}^2(\bftheta_k,\nbar) \geq \frac{\kappa_{\sigma}}{L_k}\sum_{\mfu\in\mfU_k} \sigma_{\mfu}^2(\bftheta_k)
    %\quad \text{i.o.}
    \right\}
    = 0,
    \notag
\end{align}
where $\sigma_{\mfu}^2(\bftheta_k)=\var(F(\bftheta_k,\XX) | \XX\in\mcX_{\mfu})$, and $\sigma_{\mfu}^2(\bftheta_k,\nbar)$ is the corresponding sample variance obtained with $\nbar$ MC draws.
We denote $\sigmamax^2 := \frac{\kappa_{\sigma}}{L_k}\sum_{\mfu\in\mfU_k} \sigma_{\mfu}^2(\bftheta_k)$ and $K_{fg}$ the iteration such that $\frac{1}{L_k}\sum_{\mfu\in\mfU_k}\hat{\sigma}_{\mfu}^2(\bftheta_k,\nbar) < \frac{\kappa_{\sigma}}{L_k}\sum_{\mfu\in\mfU_k} \sigma_{\mfu}^2(\bftheta_k)$ happens for all $k\geq K_{fg}$.
When the support of $\XX$ is bounded, the sample variance is upper bounded by a constant in a deterministic way \citep{popoviciu1935equations}.

Now, set $K_{\sigma}=\max(K_{\rho,mge}, K_{fg})$, where $K_{\rho,mge}$ is defined in \eqref{eq:Krhomge}, and set a small enough $\epsilon>0$ such that $K_{\sigma}<T_{\epsilon}$, where $T_{\epsilon}$ is introduced in \eqref{eq:T_epsilon}.
Note further that for all $k$ such that $K_{\sigma}\leq k<T_{\epsilon}$, we have that $\Delta_k\geq\kappa_{l\Delta} \epsilon$, for some constant $\kappa_{l\Delta}>0$ similarly to the proof of Lemma \ref{lemma:iteration_complexity} (see also Lemma 4.5 in \citealp{ha2025complexity}).

% Furthermore, by effect of the almost sure convergence (Corollary \ref{cor:convergence_of_the_algo}) and the trust-region update rule ( \eqref{eq:trust_region_update}), we can lower-bound the trust-region radius as
% \giovanni{instead of $\Deltamin$ write something depending on $\epsilon$, e.g., $\Delta_{T_{\epsilon}}$}
% \begin{align}
%     0 < \Deltamin \leq \Delta_k,
%     \qquad \forall k=0,1,\ldots,T_\epsilon,
% \end{align}
% where $T_\epsilon$ is defined in  \eqref{eq:T_epsilon} (see also Theorem 4.5, Chapter 4, in
% \citealp{wright1999numerical} and Lemma 3.12 in \citealp{fletcher2002global}), where $\Deltamin$ has order of magnitude of at least $\mcO(\epsilon)$.

Then, using the adaptive sample size in   \eqref{eq:adaptive_sampling_rule}, it holds that
\begin{align}
    W_{\epsilon} 
    & = \underbrace{\sum_{k=0}^{K_{\sigma}-1} \left( \sum_{i=0}^p N_k^{(i)} + N_k^s \right)}_{:=Q_w} + \sum_{k=K_{\sigma}}^{T_\epsilon} \left( \sum_{i=0}^p N_k^{(i)} + N_k^s \right)
    \notag
    \\
    & \leq Q_w + \kappa_N (T_{\epsilon}-K_{\sigma}+1) (p+2) \max\{\sigma_0, \sigma_{\text{MAX}}\} \lambda_{T_{\epsilon}} \Delta_{T_{\epsilon}}^{-2\gamma}
    \label{eq:sampling_rule_satisfaction}
    \\
    & \sim T_{\epsilon} \lambda_{T_{\epsilon}} \Delta_{T_{\epsilon}}^{-2\gamma},
    \label{eq:sample_complexity_lambda_delta}
\end{align}
almost surely, as $Q_w$ and $\sigmamax$ are finite almost surely, and where $\kappa_N>1$ is sufficiently large to ensure that \eqref{eq:sampling_rule_satisfaction} satisfies the sampling rule \eqref{eq:adaptive_sampling_rule}.
Then, using the values of $\lambda_k,\gamma$ in \eqref{eq:lambda_gamma}-\eqref{eq:lambda_gamma_mu_subgaussian_h_exp}, the iteration complexity result \eqref{eq:iter_complexity}, and relation $\Delta_k\geq \kappa_{l\Delta}\epsilon$ up to $T_{\epsilon}$, we verify the statement.
%\hfill\halmos
\end{proof}

% Then, using the adaptive sample size in   \eqref{eq:adaptive_sampling_rule} with $\lambda_k,\gamma$ chosen as in \eqref{eq:lambda_gamma}, it holds that
% \begin{align}\label{eq:Zk_bound}
%     W_k &= (p+1) \sum_{l=0}^{k} N_l \\
%     & \leq (2d+1) (k+1) N_{k} \\
%     & \leq (2d+1) (k+1) \max(\nbar, 2^{q}) \magenta{\hat{\sigma}_{\text{MAX}}} \sqrt{\nbar} \kappa_{as}^{-2} \Deltamin^{-2\gamma} \lambda_{k} \\
%     & = (2d+1) (k+1) \max(\nbar, 2^{q}) \magenta{\hat{\sigma}_{\text{MAX}}} \sqrt{\nbar} \kappa_{as}^{-2} \Deltamin^{-4q/(q+2)} k^{q(1+\delta)/(q+2)},
% \end{align}
% where the term $\max(\nbar, 2^{q})$ is due to the requirement that $N_k\in\mcN_q$ (see  \eqref{eq:set_N}), the choice of $\lambda_k$ and $\gamma$ follows from  \eqref{eq:lambda_gamma}, and $\delta>0$ can be chosen arbitrarily small.

% When $k=T_\epsilon$, we further have that $k=\mcO(\epsilon^{-2})$ by Lemma \ref{lemma:iteration_complexity}.
% Thus, the order of magnitude of the total number of MC samples needed to reach $\epsilon$-first-order-convergence is
% \begin{align}
%     W_{T_\epsilon}
%     & = \mcO\left( \Deltamin^{-4q/(q+2)} T_\epsilon^{(2q+2+q\delta)/(q+2)} \right) 
%     \label{eq:ZTeps_complexity}
%     \\
%     & = \mcO \left( \epsilon^{-\frac{8q+4+2q\delta}{q+2}} \right),
%     \label{eq:ZTeps_complexity_2}
% \end{align}
% which concludes the proof.

We note that, considering the five specifications \eqref{eq:sample_complexity}-\eqref{eq:sample_complexity_mu_subgaussian_h_exp} as a whole, the sample complexity of the SASTRO-DF ranges between approximately $\mcO\left(\epsilon^{-4}\right)$ (one dimensional case, bounded $\mcX$) to $\mcO(\epsilon^{-8})$. 
%(infinite dimensional case).
The following corollary shows that the sample complexity is superior to the no-stratification strategy for any fixed dimension $q<\infty$.

\begin{corollary}\label{corollary:sample_complexity_comparison}
The number of oracle calls required to reach an $\epsilon$-first-order stationary point, i.e., $\nabla f(\bftheta_k)\leq\epsilon$, under a no-stratification strategy satisfies, for sufficiently small $\epsilon>0$,
\begin{align}\label{eq:sample_complexity_NS}
    W_\epsilon^{\text{NS}}=\mcO\left( \epsilon^{-(8+2\delta)} \right),
\end{align}
almost surely, where $\delta>0$ is arbitrarily small.
\end{corollary}

\begin{proof}
The proof follows the same procedure of Theorem \ref{prop:sample_complexity}, where the no-stratification values $\lambda_k^{\text{NS}},\gamma^{\text{NS}}$ in \eqref{eq:lambda_gamma_NS} are plugged into  \eqref{eq:sample_complexity_lambda_delta}, yielding the result.
%\hfill\halmos
\end{proof}

%%%%%
\subsection{Comparison with ASTRO-DF}\label{sec:comparison}
\noindent
Theorem \ref{prop:sample_complexity} proves a reduced sample complexity with respect to the ASTRO-DF in \citep{ha2025complexity} for small dimensions $q$ of the random vector $\XX$ and for several specifications of the map $\mu$ and the oracle $F(\cdot, \XX)$ as presented in Definitions \ref{def:mu_subexp_subgaussian}-\ref{def:F_poly_exp}.
The sample complexity of ASTRO-DF is proven under the assumption that $F(\cdot,\XX)$ is sub-exponentially distributed, and it reads 
\begin{align}\label{eq:sample_complexity_ASTRO_DF}
    \mcO \left(
    \epsilon^{-6} \log(\epsilon^{-1})^{1+\delta}
    \right)
    =
    \tilde{\mcO} \left(
    \epsilon^{-6}
    \right)
    ,
    \quad \delta>0,
\end{align}
when only zeroth-order information is available and common random numbers are not used (consistently with the assumptions of our paper). 
Such result is due to a parsimonious logarithmic specification of $\lambda_k$, which ensures the almost sure convergence of ASTRO-DF by using a Bernstein-type inequality bound to cap the stochastic error for large $k$. Specifically, they set
\begin{align}\label{eq:lambda_gamma_NS_Bernstein}
    \lambda_k = \log(k)^{1+\delta},
    \quad \gamma = 2,
    \quad \forall\delta>0.
\end{align}

In our case, we use the more conservative Chebyshev inequality in Theorem \ref{prop:stochastic_error}, which allows us to (i) exploit the variance reduction technique enforced by stratified sampling, and to (ii) relax the sub-exponential assumption of $F(\cdot,\XX)$, which under some specifications of $\mu$ and $F$ is not guaranteed.
For example, if $\XX$ has mutually independent sub-exponential components (i.e., $\mu_i,\forall i,$ is sub-exponential in the sense of Definition \ref{def:mu_subexp_subgaussian}), then $F(\cdot,\XX)$ is sub-exponential if it is linear in $\XX$, but it may not be if it has general polynomial or exponential relation to $\XX$ (see Definition \ref{def:F_poly_exp}).

Although Chebyshev inequality is more conservative than Bernstein-type inequalities, by Theorem \ref{prop:sample_complexity} we know that the SASTRO-DF enjoys a sample complexity that is lower or approximately equal to \eqref{eq:sample_complexity_ASTRO_DF} under the specifications of $F(\cdot,\XX)$ and $\mu$ and dimensions $q$ reported in Table \ref{tab:comparison}.
Besides the sample complexity derived in \citep{ha2025complexity} (named $\NS_B$), in the table we also report the sample complexity of the ASTRO-DF without sub-exponential assumption of $F(\cdot,\XX)$ (see Corollary \ref{corollary:sample_complexity_comparison}), that is, a no-stratification strategy where the stochastic error is bounded via Chebyshev inequality; we remark that in this cases, our dynamics stratification produces an improved complexity under all the analyzed specifications.

\begin{table}[ht]
\caption{
    Approximate sample complexity of the SASTRO-DF for small dimensions $q$ of the underlying random vector $\XX$ and corresponding sample complexity of the no-stratification strategy.
    NS$_C$ (NS$_B$): No-stratification strategy in which the stochastic error is bounded via Chebyshev (Bernstein) inequality.
    NS$_B$ is derived in \citep{ha2025complexity} assuming sub-exponential $F(\cdot,\XX)$.
    Small terms $\delta,\varrho>0$ in the exponent of $\epsilon$ are removed for simplicity, as well as logarithmic functions of $\epsilon$ (see \eqref{eq:sample_complexity}-\eqref{eq:sample_complexity_mu_subgaussian_h_exp}, \eqref{eq:sample_complexity_NS}, \eqref{eq:sample_complexity_ASTRO_DF}).
    $^*$Here we assume the lowest sample complexity according to the values of $\kappa_{\min}$ and $b_{\max}$ (see  \eqref{eq:sample_complexity_mu_subexp_h_exp}), i.e., when $\kappa_{\min}\gg 2b_{\max}$.
    }
\renewcommand{\arraystretch}{1.5}
    \begin{tabular}{llll}
    \hline
    Cases & SASTRO-DF & NS$_C$ & $\text{NS}_B$
    \\ \hline
    $\|\nabla \mu^{-1}\|<\infty, \ \forall \uu\in[0,1]^q$
    & $\epsilon^{-4}, \epsilon^{-5},
    %\epsilon^{-5\frac{3}{5}},
    \epsilon^{-5.6},
    \epsilon^{-6} \ (q=1,\ldots,4)$
    & $\epsilon^{-8}$
    &  \multicolumn{1}{|c|}{$\epsilon^{-6}$}
    %& \multirow{1}{*}{$\epsilon^{-8}$}
    %& \multirow{1}{*}{$\epsilon^{-6}\log(\epsilon^{-1})$}
    \\ \cline{4-4}
    if $\mu_i$ is \eqref{eq:SE} $\forall i=1,\ldots,q$; $F$ is \eqref{eq:PO}
    & $\epsilon^{-5}, \epsilon^{-6} \ (q=1,2)$ & \vdots &
    \\
    if $\mu_i$ is \eqref{eq:SE} $\forall i=1,\ldots,q$; $F$ is \eqref{eq:EX}
    & $\epsilon^{-5}, \epsilon^{-6} \ (q=1,2)\ ^*$ & \vdots &
    \\
    if $\mu_i$ is \eqref{eq:SG} $\forall i=1,\ldots,q$; $F$ is \eqref{eq:PO}
    & $\epsilon^{-5}, \epsilon^{-6} \ (q=1,2)$  & \vdots &
    \\
    if $\mu_i$ is \eqref{eq:SG} $\forall i=1,\ldots,q$; $F$ is \eqref{eq:EX}
    & $\epsilon^{-5}, \epsilon^{-6} \ (q=1,2)$  & \vdots &
    \\ \hline
    \end{tabular}
    
    % $\bftheta$ & $d$-dim decision vector
    % & $F$ & Noisy obj.\ func.\
    % & $N$ & \#samples
    % & $\kappa,\xi,\bsa,\bsb$ & Parameters of $\mu$ or $F$
    % \\
    % $\XX$ & $q$-dim random vector
    % & $\mu$ & Map $\mcX\to\mcU$
    % & $L$ & \#strata
    % & $\epsilon$ & Optimality tolerance
\label{tab:comparison}
\end{table}

\section{Implementation}\label{sec:applications}
\noindent
In this section we report Algorithms \ref{alg:sastro-df} and \ref{alg:aux}, which describe the SASTRO-DF method under the stratification and no-stratification strategies.
Besides the stratification of the random vector, the distinguishable elements corresponding to each of the analyzed specification of $F(\cdot,\XX)$ and $\mu$ comes from the choice of $\lambda_k$ and $\gamma$, which follow \eqref{eq:lambda_gamma}-\eqref{eq:lambda_gamma_mu_subgaussian_h_exp} and \eqref{eq:lambda_gamma_NS} in such a way to enforce almost sure $\epsilon$-convergence for sufficiently large sample size and iteration.
In addition, we discuss applications in which the SASTRO-DF can be a suitable candidate solver, and provide numerical examples to test the efficiency of our methodology.
Moreover, we provide a heuristic data-driven sampling scheme that will be discussed in Section \ref{sec:data_driven}, where the related Algorithm \ref{alg:aux_DM} may be chosen in place of Algorithm \ref{alg:aux} when simulating from a given distribution of $\XX$ is expensive.

\begin{algorithm}
\small
\caption{SASTRO-DF algorithm. $d,q$: dimensions of $\bftheta, \XX,$ respectively (see  \eqref{eq:opt}).
Sampling strategies: stratification (STRAT), no-stratification (NS), discrete-mapping (DM; see Section \ref{sec:data_driven}).
%For ease of notation, the iteration subscript $k$ is removed when possible.
%subscripts are removed when possible (e.g., $\lambda,\gamma$ refer to the $\lambda_k\gamma$ of \eqref{eq:lambda_gamma}-\eqref{eq:lambda_gamma_mu_subgaussian_h_exp},\eqref{eq:lambda_gamma_NS}).
}
\label{alg:sastro-df}
\begin{algorithmic}[1]

\item \textbf{Input:} 
$\bftheta_0\in\RR^d$: initial solution.
$\Delta_0>0$: initial trust-region radius.
$\Deltamax>0$: maximum trust-region radius.
$\eta > 0$: success threshold.
$\tilde{\eta}>0$: scaling constant of the model gradient sufficient size.
$\ovgamma (\ungamma)$: increase (decrease) factor of the trust-region radius (with $0<\ungamma<1<\ovgamma$).
$\sigmamin^2>0$: minimum variance for the sampling rule.
$\Kmax\in\mathbb{N}$: maximum number of iterations.
$\Wmax\in\mathbb{N}$: maximum total number of samples.
%\magenta{$\Nmax\in\mathbb{N}$: maximum number of samples per evaluation.}

\If{STRAT}
\State Set $\nbar\in\mathbb{N}$ (number of samples per strata).
\State Set $\mcN=\{n\in\mathbb{N} \ : \ \frac{n}{\nbar}=\ell,\ \ell^{1/q}\in\mathbb{N}\}$ (set of candidate number of samples per evaluation).
\item \textbf{elseif} NS or DM \textbf{then} \State Set $\nbar=n_k^{(i)}$, $\forall k\geq0,\ i=0,1,\ldots,p,$ and $\mcN=\mathbb{N}$.
\EndIf

\item \textbf{initialization:} Set $w_k = n_k = n_0$, $k = 0$, $\Delta_k=\Delta_0$, $\bftheta_k=\bftheta_0$.

\State Set  $n_0\in\mcN$ (initial number of samples) and $\ell_0=n_0/\nbar$ (number of strata).

\While{$w_k < \Wmax$}

\While{$k < \Kmax$}
\State If STRAT, set 
$\lambda_k,\gamma$ as in \eqref{eq:lambda_gamma}-\eqref{eq:lambda_gamma_mu_subgaussian_h_exp}.
%depending on $F(\cdot,\XX),\mu$. 
If NS, set $\lambda_k,\gamma$ as in \eqref{eq:lambda_gamma_NS}.
If DM, set $\lambda_k,\gamma$ arbitrarily.
\label{algline:lambda_gamma}

\State Set interpolation points $\bftheta_k^{(i)},\ i=1,\ldots,p,$ following Definition \ref{def:polynomial_model_DQ}

\For{$i=0,1,\ldots,p$} \label{algline:for_start}

\State Set sample size $n_k^{(i)}=\min\{n\in\mcN \ : n\geq\lambda_k \ \}$ and number of strata $\ell_k^{(i)}=n_k^{(i)}/\nbar$.

\State If STRAT or NS, find the sample variance $\tilde{s}_k^{(i)}$ via Algorithm \ref{alg:aux}.
%with inputs $l, \bftheta$ , 
If DM, via Algorithm \ref{alg:aux_DM}.
\label{algline:estimated_stat}

\While{$n_k^{(i)}$ does not satisfy the sampling rule \eqref{eq:adaptive_sampling_rule} given $\tilde{s}_k^{(i)}$}
\State Set $\tilde{n}=n_k^{(i)}$.
\State Increase sample size $n_k^{(i)}=\min\{m\in\mcN \ : \ m>\tilde{n}\}$ and number of strata $\ell_k^{(i)} = n_k^{(i)}/\nbar$.
\State If STRAT or NS, find $\tilde{s}_k^{(i)}$ via Algorithm \ref{alg:aux}.
%with inputs $l, \bftheta$ 
If DM, via Algorithm \ref{alg:aux_DM}. 
\label{algline:estimated_stat_2}
\EndWhile

 \State If STRAT or NS, find $\tilde{f}_k^{(i)}$ via Algorithm \ref{alg:aux}.
 If DM, via Algorithm \ref{alg:aux_DM}.
 %with inputs $l$ and $\bftheta^{(i)},

 \EndFor \label{algline:for_end}

 \State Construct $M_k(\boldsymbol{\vartheta}), \ \boldsymbol{\vartheta}\in\mcB(\bftheta_k^{(0)},\Delta_k),$ using $\tilde{f}_k^{(i)}, \ i=0,1,\ldots,p$, as in Definitions \ref{def:polynomial_model}, \ref{def:polynomial_model_DQ}.
 %(or \citealp{ha2025iteration}).

 \State Compute candidate $\bftheta_k^s = \argmin_{\boldsymbol{\vartheta} \in \mcB(\bftheta_k^{(0)},\Delta)} M(\boldsymbol{\vartheta})$. 
 %where $\mcB(\bftheta,\Delta)=\{\boldsymbol{\vartheta} \ : \ \| \boldsymbol{\vartheta}-\bftheta \| \leq \Delta \}$.

 \State Find sample size $n_k^s$ and sample mean $\tilde{f}_k^s$ as in Lines \ref{algline:for_start}-\ref{algline:for_end}.
 %with inputs $l$ and $\bftheta'$.
 
 \State Compute the success ratio $\rho_k = (\tilde{f}_k^{(0)} - \tilde{f}_k^s)/(M_k({\bftheta^{(0)}}) - M_k(\bftheta^s))$ as in \eqref{eq:trust_region_rho}.

 \State If $\rho_k > \eta$ and $\Delta_k \leq \tilde{\eta} \|\nabla M_k(\bftheta_k^{(0)})\|$, set $\bftheta_{k+1}^{(0)} = \bftheta_k^s$, $\Delta_{k+1} = \min \{\ovgamma\Delta_k, \Delta_{\max}\}$. Otherwise set $\Delta_{k+1} = \ungamma\Delta_k$.

 \State Set $ k = k + 1$ and $w_k = w_k + \sum_{i=0}^{p} n_k^{(i)} + n_k^s$.
 
\EndWhile

\EndWhile

\item \textbf{output:} Return optimal solution $\bftheta_k^{(0)}$ and estimated optimum $\tilde{f}_k^{(0)}$.
%with Algorithm \ref{alg:aux}. 
%with inputs $l, \bftheta$. 
\end{algorithmic}
\end{algorithm}

\begin{algorithm}
\caption{Calculations of sample statistics under the SASTRO-DF (set $\ell=1$ for no-stratification scheme).
$^*$Line \ref{algline:generation}: uniform samples from the previous evaluations can be stored, properly rescaled, and used again, while generating only the required additional samples.
}
\label{alg:aux}
\begin{algorithmic}[1]
\item \textbf{Input:} $\ell\in\mathbb{N}$: number of strata. 
$\nbar\in\mathbb{N}$: number of samples per strata. 
$\bftheta\in\RR^d$: current solution.
$F:\Theta \times \mcX \to \mbR$: function as defined in  \eqref{eq:opt}.
$\mu:\mcX\to\mcU$: map between $\mcX$ and $\mcU$ as in Assumption \ref{ass:map}.

\State Compute $\mfL$ as the set of left endpoints in $(0,1]^q$ with $\ell^{1/q}$ splits per dimension (see  \eqref{eq:hypercube_strata_leftend}).

\State \textbf{initialization:} $\tilde{f}=0, \ \tilde{s}=0$.
\For{$\uu \in \mfL$}
 \State $^*$Generate $q$-dimensional vectors $\vv^{[m]}, m=1,\ldots,\nbar$ with Unif(0,1)-distributed components.
 \label{algline:generation}
 \State Compute $\uu^{[m]} = \uu + \vv^{[m]}/\ell^{1/q}, \ m=1,\ldots,\nbar.$
 \State Compute $\xx^{[m]} = \mu^{-1}(\uu^{[m]}), \ m=1,\ldots,\nbar$.
 \State Compute $\hat{f}_{\mfu} = \frac{1}{\nbar} \sum_{m=1}^{\nbar} F(\bftheta,\xx^{[m])})$ and update $\tilde{f} = \tilde{f}+\hat{f}_{\mfu}$.
 \State Compute $\hat{\sigma}^2_{\mfu} = \frac{1}{\nbar-1} \sum_{m=1}^{\nbar} [F(\bftheta,\xx^{[m]}) - \hat{f}_{\mfu}]^2$ and udpate $\tilde{s} = \tilde{s}+\hat{\sigma}^2_{\mfu}/\ell$.
\EndFor
\item \textbf{output:} $\tilde{f}, \tilde{s}$.
\end{algorithmic}
\end{algorithm}

\paragraph{Model fitting}
\noindent
Interesting applications of our adaptive sampling-based optimization method can be found in the context of model fitting.
In machine learning, for example, one can be interested in approximating a computationally expensive function through an artificial neural network in order to get quick real-time evaluations.
This is a crucial objective, for example, in the option pricing context, where a trader can be interested in quickly pricing exotic products and computing the sensitivities with respect to the associated risk factors.
A non-exhaustive list of these applications can be found in \citep{ruf2020neural}.
In these cases, datasets are typically synthetic due to a lack of data of certain types of financial assets, so inputs and output are generated from a given distribution and their amount is arbitrarily large, consistently with our framework.
Another area of application of adaptive sampling optimization schemes can be found in the context of energy model calibrations, as supported by recent literature such as \citep{jain2023wake,jain2024simulation}. 
More generally, our SASTRO-DF may be used to learn a system $\psi:\RR^{d}\times \RR^{q_I}\to\RR^{q_O}$ describing the relation between an input vector $\XX^I\in\RR^{q_I}$ and an output vector $\XX^O\in\RR^{d_O}$, through a coefficient vector $\bftheta\in\RR^{d}$.

To test our method, we consider three toy examples respectively given as
\begin{align}
    & \min_{\bftheta}\ 
    \mbE[ \|\bftheta\|^2+2X],
    \label{eq:ex1} \tag{Ex1}
    \\
    & \min_{\bftheta}\
    \mbE[ \|\bftheta\|^2(1+X)],
    \label{eq:ex2} \tag{Ex2}
    \\
    & \min_{\bftheta}\
    \mbE[ \|X-\bftheta\|^2],
    \label{eq:ex3} \tag{Ex3}
\end{align}
where $\bftheta\in\RR^2$, and $X$ is a truncated one dimensional standard Gaussian random variable, where for the truncation we follow Corollary \ref{cor:truncation} and set a truncation range of $\left[-5,5\right]$, i.e., five standard deviations away from the mean.
Here note that \eqref{eq:ex3} may be considered as a fitting problem in which the model $\psi(\bftheta,\XX^I)=\bftheta$ is expected to learn the noisy output $\XX^O=X$.

We study the efficiency of our algorithm by solving problems \eqref{eq:ex1}-\eqref{eq:ex3} and monitoring the objective function mean values, and corresponding 80\% confidence band, as a function of the number of objective function (noisy) instances $W_k$ (see Equation \eqref{eq:total_samples_PRELIMINARIES}), over a set of 20 runs.
In our test, we set the fixed sample size per strata as $\nbar=2$ and accordingly denote the stratification-based algorithm as SASTRODF-2 in Figure \ref{fig:toys}.
As a benchmark method, we employ the ASTRO-DF with $\lambda_k$ and $\gamma$ set as in Equation \eqref{eq:lambda_gamma_NS}, consistently with convergence results in which Chebyshev inequality is used (see a discussion in Section \ref{sec:comparison}); we denote this algorithm as ASTRODF-C in Figure \ref{fig:toys}.
Moreover, we also employ an analogous stochastic trust-region derivative-free algorithm that does not involve neither adaptive sampling nor stratified sampling, which we denote as TRODF in the figure.
We note that SASTRODF-2 outperforms the benchmark algorithms for all cases \eqref{eq:ex1}-\eqref{eq:ex3}, which supports the theoretical results reported in Section \eqref{sec:complexity} and confirms the importance of adaptively increasing the objective function accuracy as the iterations grow.

\begin{figure}
    \centering
    \includegraphics[width=0.47\linewidth]{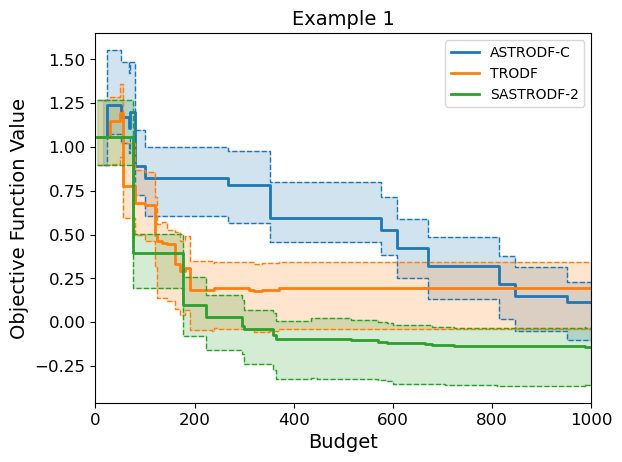}
    \includegraphics[width=0.44\linewidth]{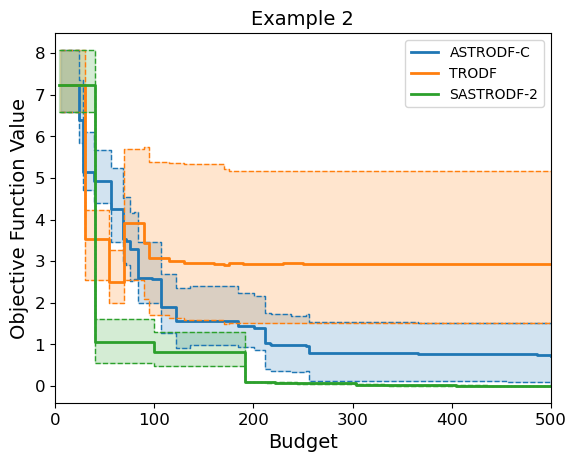}
    \includegraphics[width=0.45\linewidth]{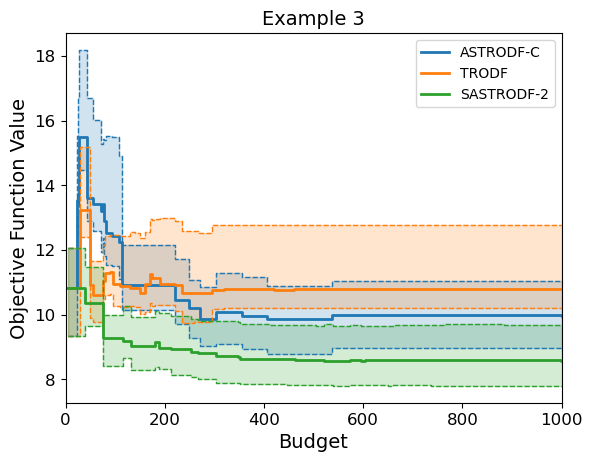}
    \caption{Mean objective function value, and 80\% confidence band, as a function of the MC sample size (denoted as \emph{Budget}).
    Optimization problems: \eqref{eq:ex1}, \eqref{eq:ex2}, \eqref{eq:ex3} (top-left, top-right, and bottom figure, respectively).
    Algorithms: SASTRODF-2, ASTRODF-C, and TRODF implemented via Algorithms \ref{alg:sastro-df}-\ref{alg:aux}.
    SASTRODF-2 computed with Algorithm \ref{alg:sastro-df}-\ref{alg:aux}, with $\nbar=2$
    and $\lambda_k,\gamma$ set as in \eqref{eq:lambda_gamma}.
    ASTRODF-C computed with Algorithm \ref{alg:sastro-df}-\ref{alg:aux}, with one fixed strata 
    $\lambda_k,\gamma$ set as in \eqref{eq:lambda_gamma_NS} (consistently with Chebyshev-like bounds, see Section \ref{sec:comparison}).
    TRODF computed with one fixed strata and fixed sample size over all iterations.
    }
    \label{fig:toys}
\end{figure}

\paragraph{Portfolio and risk management}
\noindent
Another area of application of our SASTRO-DF can be found in the context of portfolio management or hedging problems.
These problems are widely designed as, for example, robust optimization models \citep{fonseca2012robust,kim2018recent,xidonas2020robust,ghahtarani2022robust} and reinforcement learning models \citep{hambly2023recent,liu2023review,wang2025survey}.
Especially in presence of financial derivatives with nonlinear payoffs and market frictions (e.g. transaction costs, liquidity restrictions, market impact), the problem may be nonconvex and with no easy access to gradients.

In the following, we provide an example in which the SASTRO-DF could be a good candidate solver.
Let $\XX\in\RR^q$ be a return vector, $\bftheta\in\RR^d$ be portfolio holdings.
Let further $g_i, i=1,\ldots,d,$ and $\overline{g}$ be payoff functions of given contingent claims, and $\overline{\theta}$ be a fixed holding on an illiquid asset that can be interpreted as a liability.
In the special case in which, say, the $i$-th instrument is a primary financial asset, we trivially have $g_i(\XX)=X_i$.
Then, a portfolio management problem can be formulated as
\begin{align}\label{eq:portfolio_general} \tag{PM}
    \min_{\bftheta} \ 
    & \ \mathbb{E} \left[
    \upsilon\left( \sum_{i=1}^d \theta_i\ g_i(\XX) - \overline{\theta}\ \overline{g}(\XX)\right)
    \right].
%\quad \textrm{s.t.} \ \ \boldsymbol{1}^{\top}\bftheta=1.
\end{align}
where $\upsilon$ is a function describing the risk preferences of the decision maker.
%where constraint $\boldsymbol{1}^{\top}\bftheta=1$ may be addressed via Lagrangian relaxation in order for  \eqref{eq:portfolio_general} to fit into our general problem \eqref{eq:opt}.
When $\overline{\theta}=0$, \eqref{eq:portfolio_general} is a standard portfolio optimization problem with no hedging of illiquid assets involved.
Popular choices of $\upsilon$ are (negative) utility functions (e.g., exponential utility \citealp{madan2007asset,hitaj2013portfolio}) or risk measures such as the conditional value at risk \citep{alexander2006minimizing,stoyanov2013sensitivity}.
In presence of market impact, the probability distribution of $\XX$ may depend on the decision vector, which is a feature that the SASTRO-DF would allow, as the map $\mu$ (see Assumption \ref{ass:map}) is allowed to change across iterations.
In case a given contingent claim has payoff piecewise differentiable (e.g., options), in order to enforce continuous differentiability (Asssumption \ref{ass:smooth}) one could use smoothers such as kernel regression techniques (see, e.g., \citealp{nadaraya1964estimating, watson1964smooth}), or in the particular case of plain vanilla European options, the Black-Scholes formula \citep{black1973pricing} with a sufficiently small time to maturity.
%set in such a way to closely mimic the payoff while ensuring numerical stability.
%Another possibility is to employ the softplus function.
%, defined as $\textrm{softplus}(x,\nu)=(1/\nu) \log (1 + e^{\nu x})$, where $\nu>0$ regulates the smoothness. %\citep{hardle1988far}, \citep{wiemann2024using}

Given this setting, we provide a numerical example in which a portfolio manager takes a long position in a futures written on a given financial asset (e.g., a stock or a commodity) and can buy or sell options contracts in order to adjust the payoff of the position according to their risk preferences, described by a CARA exponential utility function.
In particular, we consider an out-of-the money put option and an out-of-the money call option written on the same underlying asset of the futures.
We let the underlying asset follow a geometric Brownian motion; accordingly, in this single-period example we generate price scenarios as $S_\tau=s_0\mre^{X_\tau}$, where $s_0$ is the initial asset price, $\tau$ is the time horizon, and $X_\tau$ is a normal random variable.
Correspondingly, we set the buying/selling price of these options to their Black-Scholes price.
Even in this test, we truncate the normal distribution as shown in Corollary \ref{cor:truncation} within a sufficiently large truncation range, so that the SASTRO-DF algorithm is provided with $\lambda_k,\gamma$ as in \eqref{eq:lambda_gamma}. 
As a smoother of the option terminal payoff, we select the Black-Scholes (BS) pricing function, denoted as $g_i^{\textrm{BS}}(x,\tau^{\textrm{BS}})$, with a time to maturity $\tau^{\textrm{BS}}$ small enough to ensure that the maximum error is lower than 1\% of the strike price $\kappa_i$.
%, and large enough to guarantee a certain level of smoothness.
Namely, we set $\tau^{\textrm{BS}} = \operatorname{argmax}_{t>0}\left\{ \max_{x\in\RR}\left\{[g_i^{\textrm{BS}}(x,t) - g_i(x)] \ / \ \kappa_i \right\}<0.01 \right\}$.
The parameters of the numerical test are summarized in Table \ref{tab:caraportfolio}.

As in the previous example, we plot the objective function mean values over a set of 20 optimization runs, and also plot the corresponding 80\% confidence ban, in Figure \ref{fig:portfolio}.
With the given specifications, Problem \eqref{eq:portfolio_general} could display multiple local minima, so we repeat our experiment for two different initial guesses, as showcases in the figure.
Observing Figure \ref{fig:portfolio}, the SASTRODF-2 is consistently outperforming ASTRODF-C and TRODF, confirming the results of the previous tests on problems \eqref{eq:ex1}-\eqref{eq:ex3}.
In addition, it demonstrates a remarkable stability of results, as shown by the narrow confidence bands of Figure \ref{fig:portfolio}.

To complement the analysis, we also make a more general comparison among algorithms, including all examples \eqref{eq:ex1}, \eqref{eq:ex2}, \eqref{eq:ex3}, and \eqref{eq:portfolio_general}, in Figure \ref{fig:solvability}.
In this plot we report the average of problems \emph{solved} (see \citealp{eckman2023diagnostic}) by each algorithm as a function of the fraction of budget (i.e., instances of objective functions) used, where we have fixed a maximum budget for all the algorithms up to which the algorithms stops.
In addition to the usual previously reported algorithms, we also test the SASTRO-DF with sample size per strata set to $\nbar=3$ and the ASTRO-DF with sub-exponential assumption of the noisy objective function (i.e., $\lambda_k,\gamma$ as found in \citep{ha2025complexity} using Bernstein-like bounds for convergence results, see Section \ref{sec:comparison}).
The two SASTRO-DF specifications ($\nbar=2,3$) clearly outperforms the given benchmarks, as showcased be Figure \ref{fig:solvability}.
In addition, we note that SASTRO-DF with $\nbar$ is generally more efficient than SASTRO-DF with $\nbar$, suggesting that smaller values $\nbar$ may be preferred in general.

\begin{table}[ht]
%\centering
\caption{Specifications of the portfolio allocation problem \eqref{eq:portfolio_general}.
$S_\tau$: asset price.
$\text{BS}_i$: Black-Scholes price of option $i$.}
\begin{tabular}{lll}
\hline
\textbf{Category} & \textbf{Notation} & \textbf{Value / Description} \\
\hline
\multicolumn{3}{l}{\textit{Preference parameters}} \\
%Utility function & $\upsilon(x)$ & $[1-\exp(-\alpha x)]/\alpha$ \\
Negative utility function & $\upsilon(x)$ & $\exp(-\alpha x)$ \\
Risk aversion coefficient & $\alpha$ & $0.8$ \\
\hline
\multicolumn{3}{l}{\textit{Market parameters}} \\
Risk-free rate & $r$ & $0.002$ \\
Drift of underlying & $\tilde{\mu}$ & $0.05$ \\
Volatility of underlying & $\tilde{\sigma}$ & $0.4$ \\
Initial asset price & $s_0$ & $1$ \\
Log-return truncated range & $\mcX$ & $\tilde{\mu}\pm 10\tilde{\sigma}$ \\
%Smoothing maturity (BS proxy) & $\tau_{\text{small}}$ & $10^{-5}$ \\
\hline
\multicolumn{3}{l}{\textit{Contract parameters}} \\
Maturity of contracts & $\tau$ & $1$ \\
Put option strike & $\kappa_{\text{put}}$ & $0.96$ \\
Call option strike & $\kappa_{\text{call}}$ & $1.07$ \\
Put payoff & $g_1$ & $(\kappa_{\text{put}} - S_{\tau})^+ - \text{BS}_1$ \\
Call payoff & $g_2$ & $(S_{\tau} - \kappa_{\text{call}})^+ - \text{BS}_2$ \\
Futures payoff & $\overline{g}$ & $S_{\tau} - s_0$ \\
\hline
\multicolumn{3}{l}{\textit{Tested algorithms (general method: \ref{alg:sastro-df}-\ref{alg:aux})}} \\
SASTRODF-2 & $\lambda_k,\gamma$ & \eqref{eq:lambda_gamma} \\
ASTRODF-C & $\lambda_k,\gamma$ & \eqref{eq:lambda_gamma_NS} \\
TRODF & & \\
\hline
\end{tabular}

\label{tab:caraportfolio}
\end{table}

\begin{figure}
    \centering
    \includegraphics[width=0.47\linewidth]{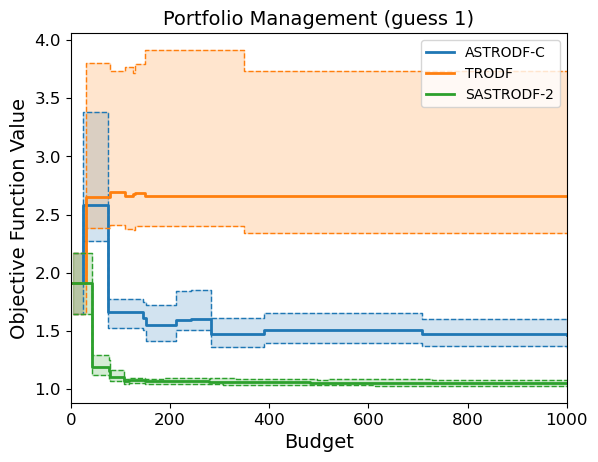}
    \includegraphics[width=0.47\linewidth]{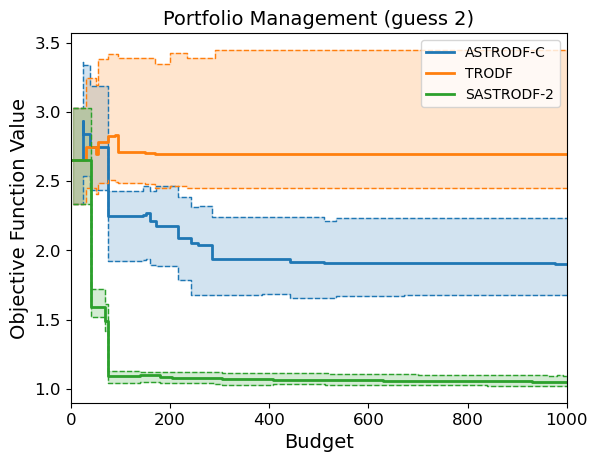}
    \caption{Objective function value as a function of the MC sample size (denoted as \emph{Budget}).
    Optimization problem: \eqref{eq:portfolio_general}.
    Two initial guesses (left and right plot, respectively).
    Algorithms: SASTRODF-2, ASTRODF-C, and TRODF implemented via Algorithms \ref{alg:sastro-df}-\ref{alg:aux}.
    SASTRODF-2 computed with Algorithm \ref{alg:sastro-df}-\ref{alg:aux}, with $\nbar=2$
    and $\lambda_k,\gamma$ set as in \eqref{eq:lambda_gamma}.
    ASTRODF-C computed with Algorithm \ref{alg:sastro-df}-\ref{alg:aux}, with one fixed strata 
    $\lambda_k,\gamma$ set as in \eqref{eq:lambda_gamma_NS} (consistently with Chebyshev-like bounds, see Section \ref{sec:comparison}).
    TRODF computed with one fixed strata and fixed sample size over all iterations.
    }
    \label{fig:portfolio}
\end{figure}

\begin{figure}
    \centering
    \includegraphics[width=0.5\linewidth]{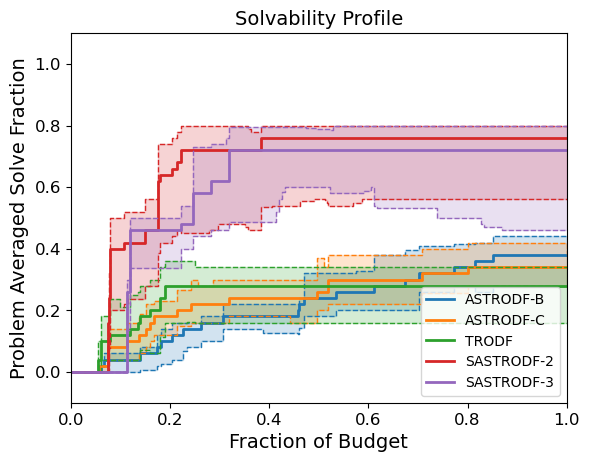}
    \caption{
    Average of problems solved by each algorithm as a function of the fraction of budget used.
    Optimization problems: \eqref{eq:ex1}, \eqref{eq:ex2}, \eqref{eq:ex3}, \eqref{eq:portfolio_general}.
    Algorithms: SASTRODF-2, SASTRODF-3, ASTRODF-C, ASTRO-B, and TRODF implemented via Algorithms \ref{alg:sastro-df}-\ref{alg:aux}.
    SASTRODF-2 (SASTRODF-3) computed with Algorithm \ref{alg:sastro-df}-\ref{alg:aux}, with $\nbar=2$ ($\nbar=3$)
    and $\lambda_k,\gamma$ set as in \eqref{eq:lambda_gamma}.
    ASTRODF-C computed with Algorithm \ref{alg:sastro-df}-\ref{alg:aux}, with one fixed strata and
    $\lambda_k,\gamma$ set as in \eqref{eq:lambda_gamma_NS} (consistently with Chebyshev-like bounds, see Section \ref{sec:comparison}).
    ASTRODF-B computed with Algorithm \ref{alg:sastro-df}-\ref{alg:aux}, with one fixed strata and
    $\lambda_k,\gamma$ set as in \eqref{eq:lambda_gamma_NS_Bernstein} (consistently with Bernstein-like bounds, see Section \ref{sec:comparison}).
    TRODF computed with one fixed strata and fixed sample size over all iterations.
    }
    \label{fig:solvability}
\end{figure}

\subsection{A data-driven alternative in high dimension}\label{sec:data_driven}
\noindent
Estimating a map $\mu$ of the type of Assumption \ref{ass:map} and sampling from it can be computationally expensive in high dimension, especially when a factor-based construction of the type shown in Corollary \ref{cor:factor_based} cannot be used to accurately describe the data.
% In addition, the estimation is subject to model errors especially when a given parametric map is considered.
% When a nonparametric map is used, the estimation error could be lower, but this typically comes at the cost of a lower interpretability.
In cases in which large amounts of high dimensional data are available, we provide an alternative data-driven method based on principal component analysis (PCA) to construct the map of interest, which we denote as $\mu_D$, where $D$ stands for \emph{discrete}.

In particular, consider a dataset composed of $\tilde{n}$ $q$-dimensional data, denoted as
\begin{align}\label{eq:dataset}
    \begin{bmatrix}
        x_{1,1} & \cdots &  x_{1,q}
        \\
        \vdots & \ddots & \vdots
        \\
        x_{\tilde{n},1} & \cdots &  x_{\tilde{n},q}
    \end{bmatrix},
\end{align}
Then, applying a PCA to the dataset \eqref{eq:dataset}, we sort the vector of scores corresponding to the first principal component in increasing order, denoted as $\bsy^{\bss}:=(y^{s_1}_{1},\ldots,y^{s_{\tilde{n}}}_{\tilde{n}})^{\top}$, where the superscript $\bss$ is the vector of original indices before sorting the vector.
Next, consider the left endpoints of the one dimensional unit interval with $\tilde{n}$ splits (as in  \eqref{eq:hypercube_strata_leftend}), i.e., $\mfL^1 := (u_1,u_2,\ldots,u_{\tilde{n}})^{\top} := (0, 1/\tilde{n},\ldots,(\tilde{n}-1)/\tilde{n})^{\top}$, and construct a discrete map
\begin{align*}
    \mu_D:
    \begin{bmatrix}
        y_1^{s_1}
        \\
        \vdots
        \\
        y_{\tilde{n}}^{s_{\tilde{n}}}
    \end{bmatrix}
    \to
    \begin{bmatrix}
        u_1
        \\
        \vdots
        \\
        u_{\tilde{n}}
    \end{bmatrix}.
\end{align*}
Then, we generate a uniform (0,1) MC samples $(u^{[1]},\ldots,u^{[\hat{n}]})^{\top}$, with $\hat{n}\leq\tilde{n}$, and find the nearest entries of $\mfL^1$ corresponding to the MC samples as
\begin{align*}
    u_{j_m} = \max_{u\in\mfL^1} \{ u\leq u^{[m]} \},
    \quad m=1,\ldots, \hat{n}.
\end{align*}
Afterwards, apply the inverse map to sample from the original dataset \eqref{eq:dataset} as
\begin{align}\label{eq:inverseMapPCA}
    \begin{bmatrix}
        x_{s_1,1} & \cdots & x_{s_1,q}
        \\
        \vdots & \ddots & \vdots
        \\
        x_{s_{\hat{n}},1} & \cdots & x_{s_{\hat{n}},q}
    \end{bmatrix}
    \leftarrow
    \begin{bmatrix}
        \mu_D^{-1} (u_{j_1})
        \\
        \vdots
        \\
        \mu_D^{-1} (u_{j_{\hat{n}}})
    \end{bmatrix}.
\end{align}

%strat possible but not so sensible, unless first PC describe most of variability
% can use other sorting methods but PCA looks more reasonable to have ≈structure

Besides avoiding the estimation of $\mu$ and having to sample from it, $\mu_D$ allows to sample any high dimensional data point by means of a one dimensional uniform sampling, as shown in  \eqref{eq:inverseMapPCA}.
%However, $\mu_D$ is not differentiable (and not even continuous), so the order of magnitude of the variance of the objective function estimate, as established in Theorem \ref{prop:var_strat}, is not directly accessible.
%Accordingly, supporting the SASTRO-DF with $\mu_D$ does not guarantee convergence, and the values of $\lambda_k$ and $\gamma$ in the sampling rule \eqref{eq:adaptive_sampling_rule} may then be chosen heuristically.
However, an almost sure bound on the stochastic error of the type of Theorem \ref{prop:stochastic_error} is not directly available, unless we assume that the data points exactly correspond to the support of the random vector and are given equal probability; accordingly, the values of $\lambda_k$ and $\gamma$ in the sampling rule \eqref{eq:adaptive_sampling_rule} may be chosen heuristically.
Alternative ways to bound the stochastic error may be inspired by the works of \citep{berry1941accuracy,esseen1942liapunov}.
%As $\mu_D$ does not attempt to describe the \emph{true} structure of data, stratification becomes obsolete in this case, for which reason the one dimensional uniform is not really stratified in \eqref{eq:inverseMapPCA}.

Moreover, we remark that the above method can potentially be supported with any sorting procedure.
Nonetheless, if the first principal component captures most of the variability in the dataset, $\mu_D$ would be able to describe at least part of the data structure, being more consistent with the typical inverse transform sampling method.
%possibly reducing the variance with respect to a general sorting method.
Associated with the procedure described in this section, we report Algorithm \ref{alg:aux_DM}, which can be used in place of Algorithm \ref{alg:aux} to support the main optimization scheme \ref{alg:sastro-df}.

% We test this scheme in a portfolio management problem of type \eqref{eq:portfolio_general}.
% In this case we assume no liability to hedge against, and we consider a portfolio of 10 primary assets.
% Although we make no assumptions on the underlying distribution, we generate a synthetic dataset coming from a multivariate normal distribution.
% We note that in this case we have an analytical solution to the optimization problem, i.e., $\bftheta^* = \alpha^{-1} \tilde{\Sigma}^{-1} \tilde{\boldsymbol{\mu}}$, which we can use as a benchmark. \giovanni{My guess: solutions will be \underline{too} different}

%in high dim we could also say to not stratify, but still here we do not rely on subexp assumption, which can be more flexible in heavy tailed cases (usual in finance for example).
%\textbf{For data-driven section:} when system does not describe perfectly the relation between inputs  an outputs and the problem is data driven, stratification can lose its power (see Chapter 5 in \citealp{cochran1977sampling}), which is why in data-driven scheme we do not address any stratification (besides the fact that we do not have structure other than the 1st PC). %...otherwise the variability of $\tilde{\epsilon}$ would be predominant when strata grows to infinity, making the two approaches equally efficient in the limit (see Chapter 5 in  \citealp{cochran1977sampling})

\begin{algorithm}
\caption{Calculations of sample statistics under the discrete map framework of Section \ref{sec:data_driven}.
Lines \ref{algline:once}, \ref{algline:once_2} are computed once for the whole Algorithm \ref{alg:sastro-df}.}
\label{alg:aux_DM}
\begin{algorithmic}[1]
\item \textbf{Input:}
$\nbar\in\mathbb{N}$: number of samples (number of strata set to 1). 
$\bftheta\in\RR^d$: current solution.
$F:\Theta \times \mcX \to \RR$: function as defined in  \eqref{eq:opt}.
$\xx_m,\ m=1,\ldots,\tilde{n}$: $q$-dimensional data.

\State Compute $\mfL$ as the set of left endpoints in $(0,1]$ with $\tilde{n}$ splits (see  \eqref{eq:hypercube_strata_leftend}).
\label{algline:once}

\State Order $\xx_m,m=1,\ldots,\tilde{n}$ with the preferred method and construct a discrete map $\mu_D$ to $\mfL$.
\label{algline:once_2}

\State Generate $u^{[m]}, m=1,\ldots,\nbar,$ from a Unif(0,1) distribution.

\State Find
$u_m=\max\{u\in\mfL \ : \ u\leq u^{[m]}\}, \ m=1,\ldots,\nbar$.

\State Compute $\xx_m = \mu_D^{-1}(u_m), \ m=1,\ldots,\nbar$.
\item \textbf{output:} $\tilde{f} = \frac{1}{\nbar} \sum_{m=1}^{\nbar} F(\bftheta,\xx_m)$ and $\tilde{s} = \frac{1}{\nbar-1} \sum_{m=1}^{\nbar} [F(\bftheta,\xx_m) - \tilde{f}]^2$.
\end{algorithmic}
\end{algorithm}

\section{Conclusions}\label{sec:conclusions}
\noindent
In this paper we develop a dynamic stratified sampling mechanism to support stochastic optimization problems with accurate evaluations of the objective function.
Our main optimization algorithm fits into the category of stochastic trust-region derivative-free methods, a recent area of research that can interest several applications in, for example, finance, engineering, and machine learning.
We show that our dynamic stratification enhance the efficiency of previous algorithms in this category.
Namely, we derive the sample complexity of our algorithm under several specifications of the objective function and show a reduced complexity with respect to pseudo random sampling strategies, previously used in this context.
We discuss applications in which our optimization algorithm can be of interest, we illustrate numerical implementations, and propose an alternative data-driven algorithm that may be used in high dimensional problems.

We envisage some interesting direction for future research.
As explained in Chapter 4.4 of \citep{glasserman2004monte}, Latin hypercube sampling (LHS) can be a more efficient method to stratify the state space in high dimension with respect to standard stratification, when the margins of the given random vector are mutually independent.
In addition, the function of the random vector (denoted as $F(\cdot,\XX)$ in our paper) should have an \emph{additive} structure with respect to its argument $\XX$ (see \citealp{stein1987large}), for LHS to perform well.
Taking into account these two conditions, it would be interesting to compare our method with the corresponding optimization algorithm equipped by a LHS strategy in terms of sample complexity and numerical experiments. 
We further note that generalizations of LHS in presence of dependence have been studied by \citep{packham2010latin}, which could extend this line of research.
Alternatively, another stratification strategy that can be of interest in high dimensional problems is post-stratification \citep{glasserman2004monte}, which is remarkably flexible as it requires little information on the data structure.

Moreover, it can be interesting to investigate the sample complexity of the SASTRO-DF when the underlying random vector follows a heavy-tailed distribution, such as distributions of L\'evy type \citep{ken1999levy}.
When $\XX$ is multidimensional, we note that factor-based constructions of these random vectors have been thoroughly explored in the last years \citep{bianchi2025welcome}, providing parsimonious calibration methods for the associated distribution functions.
In this context, computing the order of magnitude of the variance of $F(\cdot,\XX)$ take advantage of Corollary \ref{cor:factor_based}, once the order of magnitude of the variance in one dimension is identified.
We further note that L\'evy-type distributions are typically know up to an inversion of the associated Fourier transform; in such cases, there exist efficient procedures to generate samples in one dimension \citep{glasserman2010sensitivity,azzone2023fast}, or in multiple dimensions, when a factor-based representation is available \citep{amici2025multivariate}.

% \magenta{Future/current research will involve bounding the stochastic error with Bernstein-type inequalities in which the sample variance plays a role.
% Namely, the sample variance should be used to recover new optimal values of $\lambda_k,\gamma$ needed to bound the stochastic error to a term of the order $\mcO(\Delta_k^2)$, for $k\to\infty$, which in turn enforces the almost sure convergence of the algorithm.
% Such new values of $\lambda_k,\gamma$ could be less conservative than those obtained via Chebyshev inequality as done in Theorem \ref{prop:stochastic_error}.}

\appendix
% \section{My Appendix}
% Appendix sections are coded under \verb+\appendix+.

% \verb+\printcredits+ command is used after appendix sections to list 
% author credit taxonomy contribution roles tagged using \verb+\credit+ 
% in frontmatter.

%\printcredits

%% Loading bibliography style file
%\bibliographystyle{model1-num-names}
\bibliographystyle{cas-model2-names}

% Loading bibliography database
%\bibliography{cas-refs}

%\bibliographystyle{informs2014}
\bibliography{refs}

%\vskip3pt

% \bio{}
% Author biography without author photo.
% Author biography. Author biography. Author biography.
% Author biography. Author biography. Author biography.
% Author biography. Author biography. Author biography.
% Author biography. Author biography. Author biography.
% Author biography. Author biography. Author biography.
% Author biography. Author biography. Author biography.
% Author biography. Author biography. Author biography.
% Author biography. Author biography. Author biography.
% Author biography. Author biography. Author biography.
% \endbio

% \bio{figs/cas-pic1}
% Author biography with author photo.
% Author biography. Author biography. Author biography.
% Author biography. Author biography. Author biography.
% Author biography. Author biography. Author biography.
% Author biography. Author biography. Author biography.
% Author biography. Author biography. Author biography.
% Author biography. Author biography. Author biography.
% Author biography. Author biography. Author biography.
% Author biography. Author biography. Author biography.
% Author biography. Author biography. Author biography.
% \endbio

% \vskip3pc

% \bio{figs/cas-pic1}
% Author biography with author photo.
% Author biography. Author biography. Author biography.
% Author biography. Author biography. Author biography.
% Author biography. Author biography. Author biography.
% Author biography. Author biography. Author biography.
% \endbio

\end{document}